\documentclass{amsart}
\usepackage{amsmath,amsfonts,amssymb,paralist}
\usepackage[margin=3cm]{geometry}
\newcommand{\maxlen}[1]{| #1 |}
\newcommand{\ntal}{\underline{\alpha}}
\newcommand{\nts}{\underline{s}}

\newcommand{\Cn}{\mathbb{C}^n}
\newcommand{\Cd}{\mathbb{C}^d}

\newcommand{\C}{\mathbb{C}}
\newcommand{\CN}{\mathbb{C}^N}

\newcommand{\N}{\mathbb{N}}
\newcommand{\Nn}{\mathbb{N}^n}

\newcommand{\dopvar}[2]{\frac{\partial #1}{\partial #2}}

\newlength{\extendaxesby}

\DeclareMathOperator{\imag}{Im} \DeclareMathOperator{\real}{Re}

\DeclareMathOperator{\ord}{ord}

\newtheorem{thm}{Theorem}
\newtheorem{lem}[thm]{Lemma}
\newtheorem{prop}[thm]{Proposition}
\newtheorem{cor}[thm]{Corollary}

\theoremstyle{definition}
\newtheorem{defn}[thm]{Definition}
\newtheorem{exa}[thm]{Example}
\newtheorem{rem}[thm]{Remark}

\def\dbl{[\hskip -1pt [}
\def\dbr{]\hskip -1pt]}
\def \Rk{\text{\rm Rk}\,}

\newcommand{\D}{\mathcal{D}}

\newcommand{\PP}{\mathcal{P}}

\begin{document}

	\title{Finite jet determination of CR mappings}
	\author{Bernhard Lamel}

	\address{Universit\"at Wien, Fakult\"at f\"ur Mathematik, Nordbergstrasse 15, A-1090 Wien, \"Osterreich}
	\email{lamelb@member.ams.org}%
	\author{Nordine Mir}
	\address{Universit\'e de Rouen, Laboratoire de Math\'ematiques Rapha\"el Salem, UMR 6085 CNRS, Avenue de
	l'Universit\'e, B.P. 12, 76801 Saint Etienne du Rouvray, France}
	\email{Nordine.Mir@univ-rouen.fr}
	\thanks{The first author was supported by the FWF, Projekt P17111.\\
	The research of the second author was supported in part by the French National Agency for Research (ANR), project RESONANCES (programmes blancs).}

	\begin{abstract}
	We prove the following finite jet determination result for CR
	mappings: Given a smooth generic submanifold $M\subset \C^N$, $N\geq 2$, that
	is essentially finite and of finite type at each of its points,
	for every point $p\in M$ there exists an integer $\ell_p$,
	depending upper-semicontinuously on $p$, such that for every
	smooth generic submanifold $M'\subset \CN$ of the same dimension
	as  $M$, if $h_1,h_2\colon (M,p)\to M'$ are two germs of
	smooth finite CR mappings with the same $\ell_p$ jet at $p$, then
	necessarily $j^k_ph_1=j_p^kh_2$ for all positive integers $k$.
	In the hypersurface case, this result  provides several
	new unique jet determination properties for holomorphic mappings
	at the boundary in the real-analytic case; in particular, 
	it provides the finite jet determination of arbitrary
	real-analytic CR mappings between real-analytic hypersurfaces in
	$\CN$ of D'Angelo finite type. It also yields a new
	boundary version of H. Cartan's uniqueness theorem: if
	$\Omega,\Omega'\subset \CN$ are two bounded domains with smooth
	real-analytic boundary, then there exists an integer $k$, depending
	only on the boundary $\partial \Omega$, such that if
	$H_1,H_2\colon \Omega \to \Omega'$ are two proper holomorphic
	mappings extending smoothly up to $\partial \Omega$ near some
	point $p\in\partial \Omega$ and agreeing up to order $k$ at $p$,
	 then necessarily $H_1=H_2$.
	\end{abstract}

\maketitle

%\numberwithin{equation}{section}
% \numberwithin{thm}{section}

\section{Introduction}\label{s:intro}

There exists a wide variety of results concerned with the rigidity
of automorphisms of a given geometric structure. In CR geometry,
one classical result of this type is given by the uniqueness
result stating that every pseudo-conformal map (or equivalently
local biholomorphic map) sending Levi-nondegenerate real-analytic
hypersurfaces of $\CN$ into each other, $N\geq 2$, is uniquely
determined by its 2-jet at any given point; this is a consequence
of the solution to the biholomorphic equivalence problem for the
class of Levi-nondegenerate hypersurfaces, obtained by E. Cartan
\cite{Ca1,Ca2} in $\C^2$ and Tanaka \cite{Ta1} and Chern-Moser
\cite{CM} in $\C^N$ for arbitrary $N\geq 2$. This result has been
the source of many recent developments and generalizations in
several directions, see e.g.\ the works \cite{Be1,Be2,Lo,BER1,BER4,BER5,BMR1,E4,kim,Travis1,ELZ1,KZ1,LM2,LM6} and also the
surveys \cite{Vi,BERbull,Rsurvey,Zsurvey,H4} for complete
references on the subject. Most of the work mentioned above is
concerned with establishing the unique jet determination property
{\em for holomorphic automorphisms}. In this paper, we are
concerned with understanding the same phenomenon for {\em finite
holomorphic mappings} (or even arbitrary CR mappings) between
generic manifolds that we allow to be of any codimension and to
have strong Levi-degeneracies. More precisely, we prove the
following theorem (see \S \ref{s:formal} for relevant definitions
and notation).

\begin{thm}\label{t:main1}
Let $M\subset \C^N$ be a smooth generic submanifold that is
essentially finite and of finite type at each of its points.  Then
for every point $p\in M$ there exists an integer $\ell_p$,
depending upper-semicontinuously on $p$, such that for every
smooth generic submanifold $M'\subset \CN$ of the same dimension
as that of $M$, if $h_1,h_2\colon (M,p)\to M'$ are two germs of
smooth finite CR mappings with the same $\ell_p$ jet at $p$, then
necessarily $j^k_ph_1=j_p^kh_2$ for all positive integers $k$.
\end{thm}

Here and throughout the paper by smooth we mean ${\mathcal C}^\infty$-smooth.
To put our main result into the proper perspective, we should mention that
Theorem~\ref{t:main1} improves the very few finite jet determination results
for finite mappings in two important different directions. Under
the same assumptions as that of Theorem~\ref{t:main1}, Baouendi, Ebenfelt and
Rothschild proved in \cite{BER5} (see also \cite{BMR1}) the finite jet
determination of finite mappings whose $k$-jet at a given point, for $k$
sufficiently large, is {\em the same as that of a given fixed finite map};
the integer $k$ does actually depend on this fixed map.

Our
result allows, on one hand,  to compare {\em arbitrary pairs of finite maps}, and cannot
be derived from the mentioned result of \cite{BER5}.
From this point of view,  Theorem~\ref{t:main1} is more natural
and satisfactory. On the other hand, our main result also provides a dependence
of the jet order (required to get the determination of the maps) on the base
point. This explicit control cannot be obtained by the techniques of \cite{BER5,BMR1} and is of
fundamental importance in order to
derive for instance Theorem~\ref{t:main2}
below.

Note that Theorem~\ref{t:main1} is new even in the case where the manifolds and
mappings are real-analytic, in which case the conclusion is that the mappings
are identical. Note also that the upper-semicontinuity of the jet order with
respect to the base point mentioned in Theorem~\ref{t:main1} was already
obtained by the authors in \cite{LM2} in the case of local biholomorphic
self-maps of real-analytic generic submanifolds of $\CN$. The proof that we
are giving of this fact in this paper has the advantage to extend to a more
general situation and to be at the same time somewhat simpler than the proof
given in \cite{LM2}.

Theorem~\ref{t:main1} offers a number of remarkable
new consequences.  The first one is given by the following finite jet
determination result for {\rm arbitrary} CR mappings between  D'Angelo finite
type hypersurfaces (in the sense of \cite{DAfintyp}). To the authors' knowledge, this result is
the first of its kind in the levi-degenerate case. (See also
Corollary~\ref{c:newcorol} below for a slightly more general version.)

\begin{cor}\label{c:cor1} Let $M,M'\subset \CN$ be  smooth real hypersurfaces of
  D'Angelo finite type. Then for every point $p\in M$, there exists a positive
  integer $\ell=\ell (M,p)$, depending upper-semicontinuously on $p$, such that
  for any pair
  $h_1,h_2\colon (M,p)\to M'$ of germs of smooth CR mappings, if
  $j_p^\ell h_1=j_p^\ell h_2$, then necessarily $j^k_ph_1=j_p^kh_2$ for all
  positive integers $k$. If in addition both
  $M$ and $M'$ are real-analytic, it follows
  that $h_1=h_2$.
\end{cor}

In another direction, a further consequence of Theorem~\ref{t:main1} is given by the
following.

\begin{thm}\label{t:main2} Let $M$ be a compact  real-analytic CR submanifold
  of $\CN$ that is of finite type  at each of its points. Then there exists a
  positive integer $k$, depending only on $M$, such that  for every
  real-analytic CR submanifold $M'\subset \CN$ of the same  dimension as that
  of $M$ and for every point $p\in M$, local smooth CR finite mappings sending
  a neighbourhood of $p$ in $M$ into $M^\prime$ are uniquely determined by
  their $k$-jet at $p$.  \end{thm}

Theorem~\ref{t:main2} follows
from the conjunction of the
upper-semicontinuity of the integer $\ell_p$ on $p$ in Theorem~\ref{t:main1}, a
well-known result of Diederich-Forn\ae ss \cite{DF2} stating that compact
real-analytic CR submanifolds of $\CN$ do necessarily not contain any analytic
disc and hence are essentially finite (see e.g.\  \cite{BERbook}) and the
combination of the regularity result due to Meylan \cite{Me1} with the recent
transversality result due to Ebenfelt-Rothschild \cite{ER1}. In the case of
local CR diffeomorphisms, Theorem~\ref{t:main2} was already obtained by the
authors in \cite{LM2}.

When both manifolds $M$ and $M'$ are compact hypersurfaces in
Theorem~\ref{t:main2}, we have the following neater
statement as an immediate consequence of Corollary~\ref{c:cor1}.

 \begin{cor}\label{c:corol} Let $M,M'\subset \CN$ be compact real-analytic
   hypersurfaces. Then there exists a positive integer $k$ depending only on
   $M$, such that for every point $p\in M$,
   local smooth CR mappings sending a
   neighbourhood of $p$ in $M$ into $M^\prime$
   are uniquely determined by their
   $k$-jet at $p$.
 \end{cor}

We note that the conclusion of Corollary~\ref{c:corol} does not hold (even for 
automorphisms) if the 
compactness assumption is dropped, as the following example shows.

\begin{exa}\label{exa:diff}\footnote{This is an adaptation of an example which 
	 appeared in \cite{ELZ1}, which grew out of a discussion at the workshop ``Complexity of mappings in CR-geometry'' at the American Institute of Mathematics in September 2006. The authors would like to take 
	this opportunity to thank the Institute for its hospitality.}
	Let $\Phi\colon\C\to\C$ be a non-zero entire function satisfying
	\[ \dopvar{^j\Phi}{z^j} (n) = 0, \quad j\leq n, \quad n \in \N,\]
	and consider the hypersurface 
	$ M \subset \C_{z_1,z_2,w}^3$ given by the equation
	\[ \imag w = \real \left(z_1 \overline{\Phi (z_2)}\right). \] 
	Then the entire automorphism 
	\[ H(z_1,z_2,w) = (z_1 + i\, \Phi (z_2), z_2, w)\]
	sends $M$ into itself, agrees with the identity up to order $n$ at each point $(0,n,0)$, $n\in\N$, but is not
	equal to the identity. This example shows that despite of the fact that local holomorphic automorphisms of $M$ are uniquely determined by a finite jet at every arbitrary fixed point of $M$ (since $M$ is holomorphically nondegenerate and of finite type, see \cite{BMR1}), a uniform bound for the jet order valid at all  points of the manifold need not exist in general, unless additional assumptions (like compactness) are added. Note also that in view of the results in \cite{ELZ1}, the above phenomenon
	cannot happen in $\C^2$. 
\end{exa}

By a classical result of H.\ Cartan \cite{Hca}, given any bounded domain $\Omega\subset \CN$,
any holomorphic self-map of $\Omega$ agreeing with the identity mapping up to order one at
any fixed point of $\Omega$ must be the identity mapping. Our last application provides a
new boundary version of this uniqueness theorem
for proper holomorphic mappings.

\begin{cor}\label{c:nbv}  Let $\Omega\subset\CN$ be a bounded domain with
  smooth real-analytic boundary.  Then there exists an integer $k$, depending
  only on the boundary $\partial \Omega$, such that for every other bounded
  domain $\Omega'$ with smooth real-analytic boundary, if $H_1,H_2\colon \Omega
  \to \Omega'$ are two proper holomorphic maps extending smoothly up to
  $\partial \Omega$ near some point $p\in \partial \Omega$ which satisfy
  $H_1(z) = H_2(z)+o ( \left|z - p \right|^k)$, then necessarily $H_1=H_2$.
\end{cor}

Corollary~\ref{c:nbv} follows immediately from
Corollary~\ref{c:corol}. The authors do not know any other analog
of H. Cartan's uniqueness theorem for arbitrary pairs of proper
maps. A weaker version of Corollary~\ref{c:nbv} appears in the
authors' paper \cite{LM2} (namely when $\Omega=\Omega'$ and one of
the map is assumed to be the identity mapping).  For other related
results, we refer the reader to the papers  \cite{BK,H5,H6}.

The paper is organized as follows. In the next section, we recall the basic
concepts concerning  formal generic submanifolds and mappings which allow us to
state a general finite jet determination result (Theorem~\ref{t:main}) in such
a context for so-called CR-transversal mappings,
and from which Theorem~\ref{t:main1} will  be derived. In \S
\ref{s:first} we give the proof of Theorem~\ref{t:main} which involves the
Segre set machinery recently developed by Baouendi, Ebenfelt and Rothschild
\cite{BER4,BER5,BERbook}. In order to be able to compare arbitrary
pairs of mappings, we have to derive a number of new properties of the mappings under
consideration, when restricted to the first Segre set.
As a byproduct of the proof, we also obtain a new sufficient condition for a CR-transversal map
to be an automorphism (Corollary~\ref{c:byproduct}). The last part of the proof, concerned with the iteration
to higher order Segre sets, is established
by a careful analysis of standard reflection identities. During the course of the proof, we also have to keep track
of the jet order needed to get the determination of the maps so that this order behaves upper-semicontinuously on base points when
 applied at varying points of smooth generic submanifolds. This is done in the formal setting by defining
 new numerical invariants associated to any formal generic submanifold; such invariants are used to provide
  an explicit jet order that behaves upper-semicontinuously on the source manifold when this latter is subject to arbitrary continuous deformations.  The proofs of the results mentioned
in the introduction are then derived from Theorem~\ref{t:main} in \S\ref{s:end}.

\section{Formal submanifolds and mappings}\label{s:formal}

\subsection{Basic definitions}\label{ss:basicdef}

For $x=(x_1,\ldots,x_k)\in \C^k$, we denote by $\C \dbl x \dbr$
the ring of formal power series in $x$ and by $\C\{x\}$ the
subring of convergent ones. If $I\subset \C \dbl x\dbr$ is an
ideal and $F:(\C_x^k,0)\to (\C^{k'}_{x'},0)$ is a formal map, then
we define the {\em pushforward} $F_*(I)$ of $I$ to be the ideal in
$\C \dbl x'\dbr$, $x'\in \C^{k'}$, $F_*(I):=\{h\in \C \dbl
x'\dbr:h\circ F\in I\}$. We also call the {\em generic rank} of
$F$ and denote by $\Rk F$ the rank of the Jacobian matrix
${\partial F}/{\partial x}$ regarded as a $\C \dbl x\dbr$-linear
mapping $(\C \dbl x\dbr)^k\to (\C \dbl x\dbr)^{k'}$. Hence $\Rk F$
is the largest integer $r$ such that there is an $r\times r$ minor
of the matrix ${\partial F}/{\partial x}$ which is not  0 as a
formal power series in $x$. Note that if $F$ is convergent, then
$\Rk F$ is the usual generic rank of the map $F$. In addition, for
any complex-valued formal power series $h(x)$, we  denote by $\bar
h (x)$ the formal power series obtained from $h$ by taking complex
conjugates of the coefficients. We also denote by ${\rm ord}\,
h\in \N\cup \{+\infty\}$ the order of $h$ i.e. the smallest
integer $r$ such that $\partial^{\alpha}h(0)=0$ for all $\alpha
\in \N^k$ with $|\alpha|\leq r-1$ and for which
$\partial^{\beta}h(0)\not =0$ for some $\beta \in \N^k$ with
$|\beta|=r$ (if $h\equiv 0$, we set ${\rm ord}\, h=+\infty$).
Moreover, if $S=S(x,x')\in \C \dbl x,x'\dbr$, we write ${\rm
ord}_{x}\,S$ to denote the order of $S$ viewed as a power series
in $x$ with coefficients in the ring $\C \dbl x'\dbr$.

\subsection{Formal generic submanifolds and normal
coordinates}\label{ss:second}

For $(Z,\zeta)\in \C^N \times \C^N$, we define the involution
$\sigma : \C\dbl Z,\zeta\dbr \to \C\dbl Z,\zeta\dbr$ by $\sigma
(f)(Z,\zeta):=\bar{f}(\zeta,Z)$.  Let $r=(r_1,\ldots, r_d)\in
\left (\C \dbl Z,\zeta \dbr\right)^d$ such that $r$ {\em is
invariant under the involution} $\sigma$. Such an $r$ is said to
define a {\em formal generic submanifold} through the origin, which
we denote by $M$, if $r(0)=0$ and the vectors $\partial_Zr_1
(0),\ldots,\partial_Zr_d(0)$ are linearly independent over $\C$.
In this case, the number $n:=N-d$ is called the {\em CR dimension}
of $M$, the number $2N-d$ the {\em dimension} of $M$ and the
number $d$ the {\em codimension} of $M$. Throughout the paper, 
we shall freely write $M\subset \CN$. The complex
space of vectors of $T_0\CN$ which are in the kernel of the
complex linear map $\partial_Zr (0)$ will be denoted by
$T_0^{1,0}M$. Furthermore, in the case $d=1$, a formal generic
submanifold will be called a {\em formal real hypersurface}. These
definitions are justified by the fact that, on one hand, if $r\in
\left (\C\{Z,\zeta\}\right)^d$ defines a formal generic submanifold
then the set $\{Z\in \CN: r(Z,\bar{Z})=0\}$ is a germ through the
origin in $\CN$ of a real-analytic generic submanifold and
$T^{1,0}_0M$ is the usual space of $(1,0)$ tangent vectors of $M$
at the origin (see e.g.\ \cite{BERbook}). On the other hand, if
$\Sigma$ is a germ through the origin of a smooth generic
submanifold of $\CN$, then the complexified Taylor series of a
local smooth vector-valued defining function for $\Sigma$ near $0$
gives rise to a formal generic submanifold as defined above. These
observations will be used to derive the results mentioned in the
introduction from the corresponding results for formal generic
submanifolds given in \S \ref{s:further}.

Given a topological space $T$, by a {\em continuous family of
formal generic submanifolds} $(M_t)_{t\in T}$, we mean the data of a
formal power series mapping
$r(Z,\zeta;t)=(r_1(Z,\zeta;t),\ldots,r_{d}(Z,\zeta;t))$ in
$(Z,\zeta)$ with coefficients that are continuous functions of $t$
and such that for each $t\in T$, $M_{t}$ defines a formal submanifold
as described above. When $T$ is furthermore a smooth submanifold and
the coefficients depend smoothly on $t$, we say that $(M_t)_{t\in
T}$ is a {\em smooth family of formal generic submanifolds}. An
important example (for this paper) of such a family is given when
considering a smooth generic submanifold of $\CN$ near some point
$p_0\in \CN$ and allowing the base point to vary. In such a case,
the smooth family of formal submanifolds is just obtained by
considering a smooth defining function
$\rho=(\rho_1,\ldots,\rho_d)$ for $M$ near $p_0$ and by setting
$r(Z,\zeta;p)$ to be the complexified Taylor series mapping of
$\rho$ at the point $p$, for $p$ sufficiently close to $p_0$.

Given a family ${\mathcal E}$ of formal generic submanifolds of $\CN$, 
a numerical invariant $\iota$ attached to the family ${\mathcal E}$ 
 and  a submanifold $M\in {\mathcal E}$, we will further say that  $\iota (M)$ {\em depends upper-semicontinuously on 
continuous deformations of $M$} if for every continuous family 
of formal generic submanifolds $(M_t)_{t\in T}$ with
$M_{t_0} = M$ for some $t_0\in T$, there exists a neighbourhood $\omega$ of $t_0$ in $T$ such that
$M_t\in {\mathcal E}$ for all $t\in \omega$ and  such that the function
$\omega\ni t \mapsto  \iota (M_t)$ is upper-semicontinuous.

Throughout this paper, it will be convenient to use (formal) {\em
normal coordinates} associated to any formal generic submanifold $M$
of $\CN$ of codimension $d$ (see e.g.\ \cite{BERbook}). They are
given as follows. There exists a formal change of coordinates in
$\CN\times \CN$ of the form $(Z,\zeta)=(Z(z,w),\bar
Z(\chi,\tau))$, where $Z=Z(z,w)$ is a formal change of coordinates
in $\CN$ and where
$(z,\chi)=(z_1,\ldots,z_n,\chi_1,\ldots,\chi_n)\in \C^n\times
\C^n$, $(w,\tau)=(w_1,\ldots,w_d,\tau_1,\ldots,\tau_d)\in \C^d
\times \C^d$ so that $M$ is defined through the following defining
equations
\begin{equation}\label{e:coord}
r((z,w),(\chi,\tau))=w-Q(z,\chi,\tau),
\end{equation}
where $Q=(Q^1,\ldots,Q^d)\in \left(\C \dbl z,\chi,\tau
\dbr\right)^d$ satisfies
\begin{equation}\label{e:normality}
Q^j(0,\chi,\tau)=Q^j(z,0,\tau)=\tau_j,\ j=1,\ldots,d.
\end{equation}
Furthermore if $(M_t)_{t\in T}$ is a continuous (resp.\ smooth)
family of formal generic submanifolds with $M=M_{t_0}$ for some
$t_0\in T$, then one may construct normal coordinates so that the
formal power series mapping $Q=Q(z,\chi,\tau;t)$  depends
continuously (resp.\ smoothly) on $t$ for $t$ sufficiently close
to $t_0$.

\subsection{Formal mappings}\label{ss:third}
Let $r,r'\in \left(\C\dbl Z,\zeta\dbr\right)^d\times \left(\C\dbl
Z,\zeta\dbr\right)^d$ define two formal generic submanifolds $M$ and
$M'$ respectively of the same dimension and let ${\mathcal I}(M)$
(resp.\ ${\mathcal I}(M')$) be the ideal generated by $r$ (resp.\
by $r'$). Throughout the paper, given a formal power series
mapping $\varphi$ with components in the ring $ \C \dbl
Z,\zeta\dbr$, we  write
$\varphi (Z,\zeta)=0$ for $(Z,\zeta)\in {\mathcal M}$ to mean that
each component of $\varphi$ belongs to the ideal $ {\mathcal
I}(M)$. Let now $H\colon (\CN,0)\to (\CN,0)$ be a formal
holomorphic map.
For every integer $k$, the $k$-jet of $H$, denoted by $j_0^kH$, is simply the usual $k$-jet at $0$
of $H$. We associate to the map $H$ another formal map
${\mathcal H}\colon (\CN \times \CN,0)\to (\CN\times \CN,0)$
defined by ${\mathcal H}(Z,\zeta)=(H(Z),\bar{H}(\zeta))$. We say
that {\em $H$ sends $M$ into $M'$} if ${\mathcal I}(M')\subset
{\mathcal H}_*({\mathcal I}(M))$ and write $H(M)\subset M'$. Note
that if $M,M'$ are germs through the origin of real-analytic
generic submanifolds of $\CN$ and $H$ is convergent, then
$H(M)\subset M'$ is equivalent to say that $H$ sends a
neighborhood of $0$ in $M$ into $M'$. On the other hand, observe
that if $M,M'$ are merely smooth generic submanifolds through the
origin and $h\colon (M,0)\to (M',0)$ is a germ of a smooth CR
mapping, then there exists a unique (see e.g.\ \cite{BERbook})
formal (holomorphic) map $H\colon (\CN,0)\to (\CN,0)$ extending
the Taylor series of $h$ at $0$ (in any local coordinate system).
Then the obtained formal map $H$ sends $M$ into $M'$ in the sense
defined above when $M$ and $M'$ are viewed as formal generic
submanifolds.

A formal map $H\colon (\CN,0)\to (\CN,0)$ sending $M$ into $M'$
where $M,M'$ are formal generic submanifolds of $\CN$ is called
{\em CR-transversal} if
\begin{equation}\label{e:crtransverse}
T_0^{1,0}M'+dH(T_0\CN)=T_0\CN,
\end{equation}
where $dH$ denotes the differential of $H$ (at $0$). We say that
$H$ is a {\em finite} map if the ideal generated by the components
of the map $H$ is of finite codimension in the ring $\C \dbl
Z\dbr$. If $M,M'$ are merely smooth generic submanifolds through
the origin and $h\colon (M,0)\to (M',0)$ is a germ of a smooth CR
mapping, we say that $h$ is CR-transversal (resp.\ finite) if its
unique associated formal (holomorphic) power series mapping
extension is CR-transversal (resp.\ finite).

Finally, given $M,M'$ two real-analytic CR submanifolds of $\CN$, $h\colon M\to M'$
  a smooth CR mapping, $k$ a positive integer and $p$ a point in $M$, we will denote by $j_p^kh$ 
  the usual $k$-jet of $h$ at $p$. Note that 
there exists a (not necessarily unique) 
formal holomorphic map $(\CN,p)\to (\CN,h(p))$ extending the power series of $h$ at $p$ whose restriction to the intrinsic complexification of $M$ at $p$ is unique (see e.g.\ \cite{BERbook}). We then say that $h$ is a finite CR mapping if the above restricted map is a finite formal holomorphic map.

\subsection{Nondegeneracy conditions for formal submanifolds and numerical invariants}\label{ss:nondeg}
A formal vector field $V$ in $\CN \times \CN$ is a $\C$-linear
derivation of the ring $\C \dbl Z,\zeta \dbr$. If $M$ is a formal
generic submanifold of $\CN$, we say that $V$ is tangent to $M$ if
$V(f)\in {\mathcal I}(M)$ for every $f\in {\mathcal I}(M)$.

A formal (1,0)-vector field $X$ in $\C_Z^N \times \C_{\zeta}^N$ is
of the form
\begin{equation}\label{1,0}
X=\sum_{j=1}^Na_j(Z,\zeta)\frac{\partial}{\partial Z_j},\quad
a_j(Z,\zeta)\in \C \dbl Z,\zeta\dbr,\ j=1,\ldots,N.
\end{equation}
Similarly, a (0,1)-vector field $Y$ in $\C_{Z}^N \times
\C_{\zeta}^N$ is given by
\begin{equation}\label{0,1}
Y=\sum_{j=1}^Nb_j(Z,\zeta)\frac{\partial}{\partial \zeta_j},\quad
b_j(Z,\zeta)\in \C \dbl Z,\zeta\dbr,\ j=1,\ldots,N.
\end{equation}

For a formal generic submanifold $M$ of $\C^N$ of codimension $d$, we
denote by $\mathfrak g_{M}$ the Lie algebra generated by the
formal (1,0) and (0,1) vector fields tangent to $M$. The formal
generic submanifold $M$ is said to be {\em of finite type} if the
dimension of $\mathfrak g_{M}(0)$ over $\C$ is $2N-d$, where
$\mathfrak g_{M}(0)$ is the vector space obtained by evaluating
the vector fields in $\mathfrak g_{M}$ at the origin of $\C^{2N}$.
Note that if $M\subset \C^N$ is a smooth generic submanifold
through the origin, then the above definition coincides with the
usual finite type condition due to Kohn \cite{Ko1} and
Bloom-Graham \cite{BG1}.

We now need to introduce a nondegeneracy condition for formal generic submanifolds, which in the real-analytic
case was already defined by the authors in \cite{LM2}. Let therefore $M$ be
a formal generic submanifold of $\CN$ of codimension $d$ and choose normal coordinates as in \S \ref{ss:second}. For every $\alpha \in \N^n$, we set $\Theta_\alpha
(\chi)=(\Theta^1_\alpha (\chi),\ldots,\Theta^d_\alpha
(\chi)):=(Q^1_{z^\alpha}(0,\chi,0),\ldots,Q^d_{z^\alpha}(0,\chi,0))$.

\begin{defn}
We say that a formal submanifold $M$ defined in normal coordinates as
above is in the class ${\mathcal C}$ if for $k$ large enough the
generic rank of the formal $($holomorphic$)$ map $\chi \mapsto
\left(\Theta_\alpha (\chi)\right)_{|\alpha|\leq k}$ is equal to $n$. If this is the
case, we denote by $\kappa_M$ the smallest integer $k$ for which
the rank condition holds.
\end{defn}

If the formal submanifold $M\not \in {\mathcal C}$, we set $\kappa_M=+\infty$. 
In \S \ref{s:first}, we will show that for a formal submanifold $M$, being
in the class ${\mathcal C}$ is independent of the choice of normal
coordinates. Further, it will also be shown that $\kappa_M \in \N \cup \{+\infty\}$ 
is invariantly attached to $M$ (see Corollary
\ref{c:blabla}). Note that if $(M_t)_{t\in T}$ is a continous family of
formal generic submanifolds (parametrized by some topological space $T$) such that $M_{t_0}=M$ for some $t_0\in T$
and $M\in {\mathcal C}$, then there exists a neighbourhood $\omega$ of $t_0$ in $T$ such that $M_t\in {\mathcal C}$
for all $t\in \omega$ and furthermore the map $\omega \ni t\mapsto \kappa_{M_t}$ is clearly upper-semicontinuous. This remark is useful to keep in mind during the proof of Theorem~\ref{t:main}  below.
Note also that the definition of the class ${\mathcal C}$ given here coincides
with that given in \cite{LM2} in the real-analytic case. We therefore refer the reader to
the latter paper for further details on that class in the
real-analytic case.  We only note here that several comparison results between the class ${\mathcal C}$
and other classes of generic submanifolds still hold in the formal category. For instance, recall that a formal
manifold is said to be {\em essentially finite} if the formal holomorphic map $\chi \mapsto
\left(\Theta_\alpha (\chi)\right)_{|\alpha|\leq k}$ is finite
for $k$ large enough. It is therefore clear that if $M$ is essentially finite, then $M\in {\mathcal C}$. As in the real-analytic case, there are also other classes of formal submanifolds that are not essentially finite  and that still belong to the class ${\mathcal C}$. We leave the interested reader to mimic in the formal setting what has been done in the real-analytic case in \cite{LM2}.

If $M$ is a smooth generic submanifold of $\CN$ and $p\in M$, we
say that $(M,p)$ is in the class ${\mathcal C}$ (resp. essentially
finite) if the formal generic submanifold associated to $(M,p)$ (as
explained in \S \ref{ss:second}) is in the class ${\mathcal C}$
(resp.\ is essentially finite).

For every formal submanifold $M\subset \CN$, we need to
 define another numerical quantity that will be
used to give an explicit bound on the number of jets needed in
Theorem \ref{t:main}.  Given a choice of normal coordinates 
$Z=(z,w)$ for $M$, we set for any  
$n$-tuple of multiindeces 
$\ntal := (\alpha^{(1)},\dots
\alpha^{(n)})$, $\alpha^{(j)}\in\Nn$, and any $n$-tuple 
of integers $\nts := (s_1,\dots,s_n)\in \left\{ 
1,\dots,d\right\}^n$
\begin{equation}
  D_M^Z (\ntal,\nts ) =
  \det 
  \begin{pmatrix}
    \dopvar{\Theta^{s_1}_{\alpha^{(1)}}}{\chi_1} & \dots &
    \dopvar{\Theta^{s_1}_{\alpha^{(1)}}}{\chi_n} \\
    \vdots & & \vdots \\
    \dopvar{\Theta^{s_n}_{\alpha^{(n)}}}{\chi_1} & \dots &
    \dopvar{\Theta^{s_n}_{\alpha^{(n)}}}{\chi_n} \\
  \end{pmatrix}.
  \label{e:Ddef}
\end{equation}
Let us write $\maxlen{\ntal} := \max \left\{ 
|\alpha^{(j)}| \colon 1\leq j\leq n\right\}$. We now
define for every integer $k\geq 1$
\begin{equation}\label{e:definnu}
\nu^Z_M (k) :=\inf \left\{ \ord D_M^Z
(\ntal,\nts) \colon \maxlen{\ntal}\leq k \right \}\in \N \cup \{+\infty\}.
\end{equation}
Note that 
for a general formal submanifold $M$, the numerical quantity 
$\nu^Z_M(k)$ depends {\em a priori} on a choice of normal coordinates
for $M$; it will be shown in \S\ref{ss:proptrans} that $\nu^Z_M (k)$ is 
in fact independent of such a choice, and thus 
is a biholomorphic invariant of $M$. 
In view of this result, we will simply write $\nu_M (k)$ for $\nu_M^Z (k)$ for every $k$. Observe also 
that  if $M\in {\mathcal C}$ then for all $k\geq \kappa_M$, $\nu_M(k)<+\infty$.

 We also define the following quantity 
\begin{equation}\label{e:nuinftdef}
  \nu_M (\infty) := \lim_{k\to\infty} \nu_M (k) = \inf_{k \in\N} \nu_{M} (k) \in \N \cup \{+\infty\},
\end{equation}
and notice that $\nu_M (\infty) = 0$ if and only if for some $k$, the 
map $\chi\mapsto \left( \Theta_{\alpha} (\chi) \right)_{|\alpha|\leq k}$
is immersive; this is equivalent to $M$ being {\em finitely nondegenerate} (for
other possible ways of expressing this condition, see e.g. \cite{BERbook}).

Given the invariance of $\nu_M(k)$ for each $k$, it is also easy to see that if 
$(M_t)_{t\in T}$ is a continuous family of generic submanifolds, 
then for every $k\in\N^*\cup \{\infty\}$, the mappings  $T\ni t\mapsto \kappa_{M_t}$ and  
$T\ni t\mapsto \nu_{M_t}(k)$ are clearly upper-semicontinuous. Hence,
the numerical quantities $\kappa_M$ and $\nu_M (k)$ for $k\in \N^* \cup \{\infty\}$ 
depend upper-semicontinuously on continuous deformations of $M$. This fact has also to be kept in mind 
during the proof of Theorem \ref{t:main} below.

\subsection{Finite type and Segre sets mappings}\label{ss:segremaps}
We here briefly recall the definition of the Segre sets mappings
associated to any formal generic submanifold as well as the finite
type criterion in terms of these mappings due to Baouendi,
Ebenfelt and Rothschild \cite{BER4}.

Let $M$ be a formal submanifold of codimension $d$ in $\CN$ given for simplicity in
normal coordinates as in \S \ref{ss:second}. Then for
every integer $j\geq 1$, we define a formal mapping $v^j\colon
(\C^{nj},0)\to (\CN,0)$ called the Segre set mapping of order $j$
as follows. We first set $v^1(t^1)=(t^1,0)$ and define inductively
the $v^{j}$ by the formula
\begin{equation}\label{e:formula}
v^{j+1}(t^1,\ldots,t^{j+1})=(t^{j+1},Q(t^{j+1},\bar
v^{j}(t^1,\ldots,t^j))).
\end{equation}
Here and throughout the paper, each $t^k\in \C^n$ and we shall
also use the notation $t^{[j]}=(t^1,\ldots,t^j)$ for brevity. Note
that for every formal power series mapping $h\in \C \dbl
Z,\zeta\dbr$ such that $h(Z,\zeta)=0$ for $(Z,\zeta)\in {\mathcal
M}$, one has the identities $h(v^{j+1},\bar v^{j})\equiv 0$ in the ring $\C \dbl
t^1,\ldots,t^{j+1} \dbr$ and $h(v^1(t^1),0)\equiv 0$ in $\C \dbl t^1\dbr$.

The following well-known characterization of finite type
for a formal generic submanifold in terms of its Segre sets mappings
will be useful in the conclusion of the proof of Theorem
\ref{t:main}.

\begin{thm}\label{t:ber} \cite{BER4} Let $M$ be a formal generic submanifold of $\CN$. Then $M$ is of finite type
if and only if there exists an integer $1\leq m\leq (d+1)$ such
that $\Rk v^k=N$ for all $k\geq m$.
\end{thm}

\section{Statement of the main result for formal submanifolds}\label{s:further}

We will derive in \S \ref{s:end} the results mentioned in the introduction from the following finite jet determination result for formal mappings between formal submanifolds.

\begin{thm}\label{t:main}
  Let $M\subset\CN$ be a formal generic submanifold of finite type which is
  in the class ${\mathcal C}$. Then there exists an integer $K$ 
  depending only on $M$ satisfying the following properties :
  \begin{enumerate}
    \item[{\rm (i)}] For every formal generic manifold $M'$ of $\CN$ 
      with the same dimension as  $M$, 
      and for any pair  $H_1,H_2\colon (\CN,0)\to (\CN,0)$ 
      of formal CR-transversal holomorphic mappings sending 
      $M$ into $M'$ it holds that if the  $K$-jets of $H_1$ and
      $H_2$ agree, then necessarily $H_1=H_2$.
    \item[{\rm (ii)}] The integer $K$ depends upper-semicontinuously
  on continuous deformations of $M$.\end{enumerate}
\end{thm}

The upper-semicontinuity of the jet order $K$ on continuous perturbations of $M$ in the above theorem is 
of fundamental importance in order to provide the upper-semicontinuity of the integer $\ell_p$ on $p$ in Theorem~\ref{t:main1} (see \S \ref{s:end} for details). We also mention here 
the following consequence of Theorem \ref{t:main} which, 
under additional assumptions on the manifolds, 
provides a finite jet determination result valid 
for pairs of {\em arbitrary} maps. In what follows, we say that a formal manifold 
$M$ of $\CN$ contains a formal curve if there exists a non-constant
formal map $\gamma \colon (\C_t,0)\to (\CN,0)$ such that 
for every 
$h\in {\mathcal I}(M)$, $h(\gamma (t),\overline{ \gamma (t)})\equiv 0$.

\begin{cor}\label{c:cor2}
Let $M, M'\subset \CN$ be a formal real hypersurfaces. 
Assume that $M\in \mathcal{C}$ and that $M'$ does
not contain any formal curve. 
Then there exists an integer $K$, depending only 
on $M$, such that for
any pair of formal holomorphic maps 
$H_1,H_2\colon (\CN,0)\to (\CN,0)$ sending $M$ into $M'$ 
it holds that if the  $K$-jets of $H_1$ and $H_2$ agree,
then necessarily $H_1=H_2$. Furthermore, the integer $K$ can 
be chosen to depend upper-semicontinuously on 
continuous deformations of $M$.
\end{cor}

\begin{proof}{Proof of Corollary~\ref{c:cor2}}
The corollary is an immediate consequence of Theorem~\ref{t:main}
by noticing  that  any formal real hypersurface that belongs to the class 
${\mathcal C}$ is necessarily of finite type and by 
using \cite[Corollary~2.4]{LM4}  
that in the setting of Corollary \ref{c:cor2}, 
 any formal holomorphic mapping $H\colon (\CN,0)\to
(\CN,0)$ sending $M$ into $M'$ is either constant or
CR-transversal.
\end{proof}

 The proof of Theorem~\ref{t:main} is given the next section.
In order to prove this theorem, we need to establish several new properties of CR-transversal maps
 {\em along the Segre variety} 
(which is done through \S\ref{ss:proptrans}--\S \ref{ss:jetdetderiv}). Since
 the maps we consider will turn out to be {\em not totally degenerate}, 
that is, their restriction to the Segre variety is of 
generic full rank, a careful analysis of the usual reflection
identities will suffice to iterate the determination property 
along higher order Segre sets (this is carried out in \S\ref{ss:iter}). The well-known finite
type criterion (given in Theorem~\ref{t:ber}) is finally used to conclude the proof of the theorem.

\section{Proof of Theorem \ref{t:main}}\label{s:first}
In this section, we use the notation and terminology introduced in \S
\ref{s:formal}. We let $M,M'$ be two formal generic submanifolds of $\CN$ with the
same codimension $d$ and {\em fix}  a choice of normal coordinates $Z=(z,w)$
(resp.\ $Z'=(z',w')$) so that $M$ (resp.\ $M'$) is defined through the power
series mapping $Q=Q(z,\chi,\tau)$ (resp.\ $Q'=Q'(z',\chi',\tau')$) given in
\eqref{e:coord}. Recall that we write
\begin{equation}\label{e:decompQ} 
  \Theta_\alpha \left( \chi \right)
  = Q_{z^\alpha} (0,\chi,0),\quad \alpha \in \N^n.
\end{equation} 
 In what follows, we use
  analogous notations for $M'$ by just adding a ``prime" to the corresponding
  objects.

For every formal map $H\colon (\CN,0)\to (\CN,0)$, we split the map
$$H=(F,G)=(F^1,\ldots,F^n,G^1,\ldots,G^d)\in \C^n\times \C^d$$ according to the
above choice of normal coordinates for $M'$. 
If $H$ sends $M$ into $M'$, we have the following
fundamental $\C^d$-valued identity 
\begin{equation}\label{e:fundamental}
  G(z,Q(z,\chi,\tau))=Q'(F(z,Q(z,\chi,\tau)),\bar{F}(\chi,\tau),\bar
  G(\chi,\tau)), 
\end{equation}
which  holds in the ring $\C \dbl z,\chi,\tau\dbr$. Note
that $H$ is CR-transversal if and only the $d\times d$ matrix
$G_w(0)$ is invertible (see e.g.\ \cite{ER1}). Recall also that
$H$ is {\em not totally degenerate} if $\Rk F_z(z,0)=n$.

 For every
positive integer $k$, we denote by 
$J^k_{0,0}(\C^N,\CN)$ the jet space
of order $k$ of formal holomorphic maps $(\CN,0)\to (\CN,0)$ and
by $j_0^k$ be the $k$-jet mapping. (After identifying 
the jet space with polynomials of degree $k$, this is 
just the map which truncates the Taylor series at degree $k$.) 
As done before, we equip the
source space $\CN$ with normal coordinates $Z$ for $M$
and the target space
$\CN$ with normal coordinates $Z'$ for $M'$. 
This choice being fixed, we
denote by $\Lambda^k$ the corresponding coordinates on $J^k_{0,0}(\C^N, \CN)$
and by ${\mathcal T}_0^k(\CN)$ the open subset of $J^k_{0,0}(\C^N,
\CN)$ consisting of $k$-jets of holomorphic maps $H=(F,G)$ for
which $G_w(0)$ is invertible. Hence, for every formal
CR-transversal mapping $H$ sending $M$ into $M'$, we have
$j_0^kH\in
{\mathcal T}_0^k(\CN)$.

\subsection{Properties of CR-transversal maps on the first Segre set}
\label{ss:proptrans}
We start
by establishing here a few facts concerning  CR-transversal formal holomorphic
mappings sending formal generic submanifolds into each other. We will in
particular derive the following list of important properties:

\begin{enumerate}
\item[(1)] we provide the invariance
of the condition to be in the class ${\mathcal C}$ for a formal submanifold $M$ as
well as the invariance of the associated numerical quantities
$\kappa_M$ and $\nu_M (k)$ for $k\in \N^*$ (Corollary \ref{c:blabla}).
\item[(2)] we obtain some rigidity properties of CR-transversal mappings
between submanifolds in the class ${\mathcal C}$, e.g.\ the fact that they are necessarily not totally degenerate
with a certain uniform bound on the degeneracy 
considered (see Corollary~\ref{c:rigid} and Equation~\eqref{e:unifbd}) 
as well as their determination on the first Segre
set by a finite jet (Corollary \ref{c:finitejetfirst}).
\item[(3)] as a byproduct of the proofs, we obtain a new sufficient condition on $M$ that force any CR transversal formal map sending $M$ into another formal submanifold $M'$ of the same dimension to be a formal biholomorphism (Corollary~\ref{c:byproduct}). 
\end{enumerate}

All the above mentioned properties will be obtained as consequences of the following result, which can be seen as a generalization in higher codimension of an analogous version obtained for the case of hypersurfaces in \cite{DF1}.

\begin{prop}\label{p:invariance} Let $M,M'$ be formal generic submanifolds of
  $\CN$ of the same dimension. Then for
  every $\alpha\in \N^n$, there exists a universal $\C^d$-valued holomorphic
  map $\Phi_\alpha$ defined in a neighbourhood of $\{0\}\times {\mathcal
  T}_0^{|\alpha|}(\CN)\subset \C^{dr_{|\alpha|}}\times {\mathcal
  T}_0^{|\alpha|}(\CN)$, where $r_{|\alpha|}:={\rm card}\, \{\beta \in \N^n:
  1\leq |\beta|\leq |\alpha|\}$,  such that for every CR-transversal formal map
  $H\colon (\CN,0)\to (\CN,0)$ sending $M$ into $M'$, we have
  \begin{equation}\label{e:start1} \Theta_{\alpha} (\chi)=\Phi_\alpha
    \left(\left(\Theta_{\beta}'(\bar F (\chi,0))\right)_{|\beta|\leq
    |\alpha|},j_0^{|\alpha|}H\right).  \end{equation} \end{prop}

\begin{proof} We proceed by induction on the length of $\alpha$. For every
  $j=1,\ldots,n$, we denote by $e_j$ the multiindex of $\N^n$ having $1$ at the
  $j$-th digit and zero elsewhere.  Let $H$ be as in the statement of the
  proposition. Differentiating \eqref{e:fundamental} with respect to $z_j$,
  evaluating at $z=\tau=0$ and using the fact that $G(z,0)\equiv 0$ (which
  follows directly from \eqref{e:fundamental}) yields
  \begin{equation}\label{e:dday} G_w(0)\cdot \Theta_{e_j}(\chi)=\sum_{k=1}^n
    \Theta_{e_k}'(\bar F(\chi,0))\, (F^k_{z_j}(0)+F^k_w(0)\cdot \Theta_{e_j}(\chi)),
  \end{equation} where $\Theta_{e_j}$ is considered as a column vector  and
  $F^k_w(0)$ as a row vector. We thus define polynomial maps 
  \[A\colon
  \C^{dn}\times J^1_{0,0}(\CN,\CN)\to \mathbb M_d (\C),
  \quad B_j \colon
  \C^{dn}\times J^1_{0,0}(\CN,\CN)\to  \C^d, \quad j = 1,\dots,n,\]
where
  $\mathbb M_d$ denotes the space of $d\times d$ complex-valued matrices, so 
  that for each
  $j=1,\ldots,n$, 
  so that for every map $H$ as above
  \begin{equation}\label{e:addnew} A((\Theta_{\beta}'(\bar F
    (\chi,0)))_{|\beta|=1},j_0^1H)=\frac{\partial G}{\partial
    w}(0)-\sum_{k=1}^n\Theta_{e_k}'(\bar F (\chi,0))\cdot F_w^k(0),
  \end{equation}
    $$B_j ((\Theta_{\beta}'(\bar F
    (\chi,0)))_{|\beta|=1},j_0^1H)=\sum_{k=1}^n\Theta_{e_k}'(\bar F(\chi,0))\,
    F^k_{z_j}(0).$$
    Note also that for all 
    $\Lambda^1 \in {\mathcal
    T}^1_0(\CN),\, {\rm det}\, A(0,\Lambda^1)\not =0$. Therefore,
    $\Phi_{e_j}:=A^{-1}\cdot {B_j}$ is holomorphic in a neighbhourhood of
    $\{0\}\times {\mathcal T}_0^1(\CN)\subset \C^{dn} \times {\mathcal
    T}_0^1(\CN)$ and satisfies the desired property in view of \eqref{e:dday}.

To prove \eqref{e:start1} for $|\alpha|>1$, one differentiates
\eqref{e:fundamental} with respect to $z^\alpha$ and evaluates at $z=\tau=0$.
Using the induction to express every term $\Theta_{\beta}$ with $|\beta|<|\alpha|$ by
$\Phi_\beta$, we obtain for every formal map $H\colon (\CN,0)\to (\CN,0)$
sending $M$ into $M'$ an expresssion of the form 
$$A((\Theta_{\beta}'(\bar F
(\chi,0)))_{|\beta|=1},j_0^1H)\cdot \Theta_{\alpha} (\chi)=B_\alpha
\left((\Theta_{\beta}'(\bar F (\chi,0)))_{|\beta|\leq
|\alpha|},j_0^{|\alpha|}H\right),$$ 
where $B_\alpha \colon
\C^{dr_{|\alpha|}}\times J^1_{0,0}(\CN,\CN)\to  \C^d$ 
is a universal polynomial
map and $A$ is given by \eqref{e:addnew}. 
As in the case of multiindices of
length one, we conclude by setting $\Phi_\alpha:=A^{-1}\cdot B_\alpha$. The
proof of Proposition \ref{p:invariance} is complete.  \end{proof}

A number of interesting consequences may be derived from Proposition
\ref{p:invariance}. For instance, it immediately yields the following
corollary;
we note that we have not yet proved the independence of the quantities $\kappa_M$ and $\nu_M(k)$ for $k\in \N^*\cup \{\infty\}$ on
the choice of coordinates; however, this invariance, stated
in Corollary~\ref{c:blabla} below, is an immediate
consequence of Corollary \ref{c:rigid}, 
so we already state this latter in the invariant way.

\begin{cor}\label{c:rigid} Let $M,M'\subset \CN$ be a formal generic submanifolds of
  the same dimension. Suppose that $M$
  belongs to the class ${\mathcal C}$ $($as defined in \S {\rm
  \ref{ss:nondeg}}$)$ and that there exists a formal CR-transversal map
  $H\colon (\CN,0)\to (\CN,0)$ sending $M$ into $M'$. Then necessarily $H$ is
  not totally degenerate, $M'\in {\mathcal C}$, $\kappa_{M'}\leq \kappa_M$, 
  and for 
  every integer $k\geq 1$, 
  $\nu_{M} (k) \geq \nu_{M'} (k) + \ord \det \bar F_\chi (\chi,0)$.
\end{cor}

\begin{proof} We start the proof
  by introducing some notation which will be used consistently from 
  now on. For any $n$-tuple of multiindeces of $\N^n$ 
  $\ntal=\left( \alpha^{(1)},\dots,\alpha^{(n)} \right)$ and integers
  $\nts = (s_1,\dots,s_n)\in\left\{ 1,\dots,d \right\}^{n}$, we write
  \begin{equation}
    \Theta_{{\ntal,\nts}}  = \left( 
    \Theta_{\alpha^{(1)}}^{s_1}, \dots ,
    \Theta_{\alpha^{(n)}}^{s_n}\right), 
    \label{e:defthals}
  \end{equation}
  and 
  \begin{equation}\label{e:defphials}
    \Phi_{{\ntal,\nts}}
    :=\left(\Phi^{s_1}_{\alpha^{(1)}},\ldots,\Phi^{s_n}_{\alpha^{(n)}}
    \right)
  \end{equation} 
  for the corresponding map 
  given by Proposition \ref{p:invariance}. We thus have 
  from the same Proposition that
  \begin{equation}\label{e:tired} \Theta_{{\ntal,\nts}}(\chi)
    =\Phi_{{\ntal,\nts}}
    \left(\left(\Theta_{\beta}'(\bar F (\chi,0))\right)_{|\beta|\leq
    \maxlen{\ntal}},j_0^{\maxlen{\ntal}}H\right)=
    \Phi_{{\ntal,\nts}}
    \left(\Theta_{\maxlen{\ntal}}'
    (\bar F (\chi,0)),j_0^{\maxlen{\ntal}}H\right),  
  \end{equation} 
  where we use the notation
  $\Theta'_k = (\Theta'_\beta)_{|\beta|\leq k}$ for every integer $k$. 
  We also write for any $\ntal$, $\nts$
  \begin{equation}\label{e:upsilon} 
    \Upsilon^H_{\ntal,\nts} (\chi'):=\Phi_{{\ntal,\nts}}
    \left(\Theta_{\maxlen{\ntal}}'(\chi'),
    j_0^{\maxlen{\ntal}}H\right),  
  \end{equation} 
  where we recall that $\Phi_{{\ntal,\nts}}=\Phi_{{\ntal,\nts}} \left(X, \Lambda^{\maxlen{\ntal}}\right)$ is holomorphic in a neighbourhood of $\{0\} \times {\mathcal T}_0^{\maxlen{\ntal}}(\CN)\subset \C^{dr_{\maxlen{\ntal}}}\times {\mathcal T}_0^{\maxlen{\ntal}}(\CN)$.  Since $M\in \mathcal{C}$, we can choose   $n$-tuples
  of multiindeces
  $\ntal$ 
  and integers $\nts$ with $\maxlen{\ntal}= \kappa_M$ 
  such that the formal map
  $\chi \mapsto \Theta_{{\ntal,\nts}}(\chi)$ is of
  generic rank $n$. 
  Differentiating
  \eqref{e:tired} with respect to $\chi$ yields
  \begin{equation}\label{e:key} 
    \frac{ \partial\Theta_{{\ntal,\nts}}}{\partial \chi}(\chi)
    =\frac{\partial \Upsilon^H_{{\ntal,\nts}}}{\partial \chi'}
    (\bar  F(\chi,0))\cdot  \bar F_\chi(\chi,0).  
  \end{equation} 
  From
  \eqref{e:key}, we immediately get that $\Rk \displaystyle \bar
  F_\chi(\chi,0)=n$ i.e.\ that $H$ is not totally degenerate. We also
  immediately get that $$\Rk \frac{\partial \Upsilon^H_{
  {\ntal,\nts}}}{\partial \chi'} (\chi')=n,$$ which implies
  in view of \eqref{e:upsilon} that the generic rank of the map 
  $\chi'\mapsto \Theta_{\kappa_M}'(\chi')$ is also
  $n$, which shows that $M'\in {\mathcal C}$ and that $\kappa_{M'}\leq
  \kappa_{M}$.  

  Let us now prove the inequality for $\nu_M$. To this end, for every integer $k\geq 1$ 
  and for every choice of $\ntal= (\alpha^{(1)}, \dots, \alpha^{(n)})\in \Nn
  \times \dots \times \Nn$ with $\maxlen{\ntal}\leq k$ 
  and $\nts=(s_1,\dots,s_n) \in 
  \left\{ 1,\dots,d \right\}^n$, we consider the resulting equation
  \eqref{e:tired}. Differentiating \eqref{e:tired}
  with respect to $\chi$ yields 
  that the determinant considered in \eqref{e:Ddef} is expressed 
  as the product of $\det \bar F_\chi (\chi,0)$ with the determinant of 
  \[\frac{\partial \Phi_{{\ntal,\nts}}}{\partial X} 
  \left(\Theta^\prime_{\maxlen{\ntal}} (\chi'),j_0^{\maxlen{\ntal}}H\right)\Big|_{
  \chi' = \bar F(\chi,0)} 
  \cdot \frac{\partial \Theta^\prime_{{\maxlen{\ntal}}}}{\partial \chi'} 
  (\bar F(\chi,0)).\]
  Applying the Cauchy-Binet formula (allowing to express the determinant
  of this matrix product as the sum of the product of corresponding minors
  of the factors), we get the equation
  \[
  D_M^Z (\ntal,\nts) = 
  \left( \sum_{
  \substack{\maxlen{\underline{\beta}}\leq k \\ 
  \underline{t} \in \left\{ 1,\dots,d \right\}^n}} 
  a_{\underline{\beta},\underline{t}} (\chi) D_{M'}^{Z'} (\underline{\beta} ,\underline{t}) 
  (\chi')\Big|_{\chi'=\bar F( \chi,0)}\right) \det \bar F_\chi (\chi,0).
  \]
  From this we 
  see that the order of the of the right hand side
  is at least $\nu_{M'} (k)  + \ord \det \bar F_\chi (\chi,0)$, and 
  since this holds for any choice of $\ntal$ and $\nts$ as above, we obtain the inequality
  $\nu_M (k) \geq \nu_{M'} (k) + \ord \det \bar F_{\chi} (\chi,0)$.
\end{proof}

\begin{rem}\label{r:order} Under the assumptions and notation of the proof of
  Corollary \ref{c:rigid}, it also follows from \eqref{e:key} 
  that the order of the
  power series
  \[\chi \mapsto {\rm det}\, \left( \displaystyle \frac{\partial \Upsilon^H_{
  {\ntal,\nts}}}{\partial \chi'}
  (\bar F(\chi,0))\right)\]
  is {\em uniformly bounded} by $\nu_M (k)$ for any choice 
  of $n$-tuple of multiindeces $\ntal$ with
  $\maxlen{\ntal} \leq k$ and of integers $\nts = (s_1,\dots,s_n)$ for
  which $\ord D_M^Z (\ntal,\nts) = \nu_M(k)$. This fact will be
  useful in the proof of Corollary \ref{c:finitejetfirst} and Proposition
  \ref{p:soso} below.  
\end{rem}
\begin{rem}
  It is easy to see that the inequality $\nu_{M'}(k) + \ord \det \bar F_\chi
  (\chi,0)\leq \nu_M(k)$ may be strict; consider for 
  example $M = \left\{(z,w)\in \C^2: \imag w = |z|^8 \right\}$, $M' = \left\{ (z,w)\in \C^2: \imag w = |z|^{4} \right\}$, 
  and $H(z,w)= (z^2,w)$. Our proof also gives the 
  somewhat better inequality 
  $$\nu_{M'}(k)\cdot \,\ord \bar F(\chi,0) + \ord \det \bar F_\chi (\chi,0)
  \leq \nu_M(k)$$
  (in which equality holds in the above example in $\C^2$, but 
  not in general). The inequality given in Corollary~\ref{c:rigid} 
  is strong enough 
  in order to derive the invariance in Corollary~\ref{c:blabla}
  below, so we will not dwell on this matter any longer.
\end{rem}

From  Corollary \ref{c:rigid}, the invariance of $\kappa_M$ and $\nu_M (k)$
immediately follows.

\begin{cor}\label{c:blabla} Let $M$ be a formal generic submanifold of $\CN$. Then
  the condition for $M$ to be in the class ${\mathcal C}$ is independent of the
  choice of $($formal$)$ normal coordinates. Moreover, for $M$ arbitrary, the integers
  $\kappa_M$ and $\nu_M (k)$ for $k\in \N^*\cup \{\infty\}$, 
  defined in \S {\rm \ref{ss:nondeg}}, are
  also independent of a choice of such
  coordinates and hence invariantly attached to the formal submanifold $M$.
\end{cor}

Another consequence that is noteworthy to point out 
is given by the following criterion for a CR-transversal map 
to be an automorphism. Note that the inequality for the numerical invariant
$\nu_M$ given in Corollary~\ref{c:rigid} implies 
that for any CR-transversal map $H$ sending the 
formal generic submanifold $M$ 
of $\CN$, where $M\in \mathcal{C}$, into another
formal generic submanifold
$M'$ of $\CN$ with the same dimension, it follows that
\begin{equation}
  \ord \det \bar F_\chi (\chi,0) \leq \nu_M (\infty).
  \label{e:unifbd}
\end{equation}
Recalling that $\nu_M(\infty)=0$ if and only if $M$ is finitely nondegenerate, we therefore get:
  \begin{cor}\label{c:byproduct}
  Let $M,M'\subset \CN$ be  formal generic submanifolds of the same dimension, and
  assume that $M\in\mathcal{C}$. Then a formal
  CR-transversal holomorphic map sending $M$ into $M'$ is an automorphism
  if and only if for some $k\geq\kappa_M$, 
  $\nu_M (k) = \nu_{M'} (k)$. Furthermore,
  if $M$ is finitely nondegenerate, every formal CR-transversal map is a formal biholomorphism.
\end{cor}
\begin{rem}
(i) A criterion analogous to the second part of Corollary~\ref{c:byproduct} for a formal {\em finite} holomorphic mapping to be a biholomorphism was obtained in
\cite[Theorem 6.5]{ER1} under the additional assumption that $M$ is of finite type. In fact, this latter result can also be seen as a consequence of Corollary \ref{c:byproduct} in conjunction with the transversality result \cite[Theorem 3.1]{ER1}. Note also that the second part of Corollary~\ref{c:byproduct} does not hold for finite maps as can be seen by considering $M=M'=\{(z,w_1,w_2)\in \C^3: \imag w_1=|z|^2,\ \imag w_2=0\}$ and $H(z,w_1,w_2)=(z,w_1,w_2^2)$.

(ii) A nice application of the preceding corollary is also a ``one-glance''
proof of the fact that (for example) the hypersurfaces
\[ M_1 \colon \imag w = |z_1|^2 + \real z_1^2 \bar z_2^3 +
\real z_1^4 \bar z_2 + O(6), \quad 
M_2 \colon \imag w = |z_1|^2 + \real z_1^2 \bar z_2^2 + \real z_1^4 \bar z_2 + O(6),
\] are not biholomorphically equivalent; indeed, both are 
finitely nondegenerate, and we have
\[ \kappa_{M_1} = \kappa_{M_2} = 2, \quad   \nu_{M_1} (k)= \nu_{M_2} (k), \text{ for } k\neq 2,
\text{ but } 2 = \nu_{M_1} (2) \neq \nu_{M_2} (2) = 1.\]
\end{rem}

As a consequence of \eqref{e:unifbd} and \cite[Corollary 2.4]{LM4}, we also get 
following  property that under some additional assumptions on the manifolds, tangential flatness up to a certain
order of a given map implies that it is necessarily constant.

\begin{cor}
  Let $M\subset \CN$ be a formal real hypersurface given in normal coordinates as above, and 
  assume that $M\in \mathcal{C}$. Then 
  there exists an integer $k$ such that for every formal real hypersurface
  $M'\subset \CN$ not containing any formal curve and every formal holomorphic map
  $H\colon (\CN,0)\to (\Cn,0)$ sending $M$ into $M'$, $H=(F,G)$ is constant if and only if
  \[ F_{z^\alpha}(0)= 0, \quad 1\leq |\alpha|\leq k.\]
\end{cor}

For the purposes of this paper, the most important 
consequence of Proposition~\ref{p:invariance} lies in  the following finite
jet determination property.

\begin{cor}\label{c:finitejetfirst} Let $M,M'\subset \CN$  
  be formal generic submanifolds of the same dimension, given in normal coordinates as above. 
  Assume that $M$ belongs to the class ${\mathcal C}$. Then the integer
  \[
  k_0:=\min_{k\geq \kappa_M}
  \max \{k,\nu_M(k)\}
  \]
  satisfies the following property:
  For any pair  $H_1,H_2\colon (\CN,0)\to
  (\CN,0)$ of formal CR-transversal holomorphic mappings sending $M$ into
  $M'$,
  if the $k_0$-jets of $H_1$ and $H_2$ agree, 
  then necessarily $H_1(z,0)=H_2(z,0)$.
  Furthermore,
  $k_0$ depends upper-semicontinuously on continuous deformations of $M$.
\end{cor}

\begin{proof} 
  Let
  $\tilde k$ be an integer with 
  $\max \left\{ \tilde k,\nu_{M}(\tilde k) \right\} = k_0$.
  We choose $\ntal = (\alpha^{(1)},\dots,\alpha^{(n)})$
  with $\maxlen{\ntal}\leq \tilde k$ and 
  $\nts = (s_1,\dots,s_n)$ such that
  $\ord D_M^Z (\ntal,\nts) = \nu_M (\tilde k)$.   
  We use the notation of the proof of Corollary \ref{c:rigid}, 
  in particular, we consider the
  function 
  $\Upsilon_{{\ntal,\nts}}^{H_j}$ defined there, 
  with this choice of $\ntal$ and
  $\nts$ and for a given pair $H_1,H_2$ 
  of formal CR-transversal maps satisfying
  $j_0^{k_0}H_1=j_0^{k_0}H_2$.
  In view of \eqref{e:upsilon}, we have
  \begin{equation}
    \Upsilon_{{\ntal,\nts}}^{H_1}(\chi')
    =\Upsilon_{{\ntal,\nts}}^{H_2}(\chi')
    =:\Upsilon_{{\ntal,\nts}}(\chi').
  \end{equation} 
  We write $H_j=(F_j,G_j)\in \C^n\times \C^d$, $j=1,2$. We now
  claim that $\bar F_1(\chi,0)=\bar F_2(\chi,0)$ which yields the desired
  result. Indeed first note that the identity 
  $$\Upsilon_{{\ntal,\nts}} (y)-\Upsilon_{{\ntal,\nts}}
  (x)=(y-x)\cdot \int_0^1\frac{\partial \Upsilon_{{\ntal,\nts}}}{\partial \chi'}(ty+(1-t)x)dt,$$ gives in view of
  \eqref{e:tired} and \eqref{e:upsilon} that
  \begin{equation} 0=(\bar
    F_2(\chi,0)-\bar F_1(\chi,0))\cdot \int_0^1
    \frac{\partial \Upsilon_{{\ntal,\nts}}}{\partial \chi'}(t\bar
    F_2(\chi,0)+(1-t)\bar F_1(\chi,0))dt. 
  \end{equation} 
  To prove the claim,
  it is therefore enough to show that 
  \begin{equation}\label{e:detnonzero}
    {\rm det} \left( \int_0^1\frac{\partial \Upsilon_{{\ntal,\nts}}}{\partial \chi'}(t\bar F_2(\chi,0)+(1-t)\bar
    F_1(\chi,0))dt\right)\not \equiv 0.  
  \end{equation} 
  By Remark
  \ref{r:order}, the order of the power series $\chi \mapsto {\rm det}\,
  \left( \displaystyle \frac{\partial \Upsilon_{{\ntal,\nts}}}{\partial \chi'}(\bar F_2(\chi,0))\right)$ is at 
  most $\nu_M (\tilde k)$ and
  since $\bar F_1(\chi,0)$ agrees with $\bar F_2 (\chi,0)$ up to order
  $k_0\geq \nu_M (\tilde k)$, it follows that \eqref{e:detnonzero} automatically
  holds. The proof of Corollary \ref{c:finitejetfirst} is complete, up 
  to the upper-semicontinuity of the integer $k_0$, which is 
  a direct consequence of the upper-semicontinuity on continuous deformations of $M$ of the numerical invariants $\kappa_M$ and $\nu_M(k)$ for all $k\in \N^*\cup \{\infty\}$.
\end{proof}

\subsection{Finite jet determination of the derivatives on the first Segre set}\label{ss:jetdetderiv}
Our next goal is to establish a finite jet determination property similar to that obtained
in Corollary~\ref{c:finitejetfirst}, but this time for the derivatives of the
maps. For this, we will need a number of small technical lemmas. In what
follows, for every positive integer $\ell$, we write $\widehat j_{\zeta}^{\ell}
\bar{H}$ for $(\partial_{\zeta}^{\alpha}\bar H (\zeta))_{1\leq |\alpha|\leq
\ell}$ and similarly for $\widehat j_{Z}^{\ell} {H}$ to mean
$(\partial_{Z}^{\alpha} H (Z))_{1\leq |\alpha|\leq \ell}$. 
 We also keep the notation introduced in previous sections. We start with the
following.

\begin{lem}\label{l:computecomplex}
Let $M,M'\subset \CN$ be formal generic submanifolds  of codimension $d$ given in normal coordinates as above.
Then for every multiindex $\mu \in \N^d\setminus \{0\}$, there exists a universal $\C^d$-valued power series mapping  ${\mathcal S}_\mu=
{\mathcal S}_\mu (Z,\zeta, Z',\zeta';\cdot)$ polynomial in its last argument with coefficients in the ring $\C \dbl Z,\zeta,Z',\zeta'\dbr$ such that for every formal holomorphic map $H\colon (\CN,0)\to (\CN,0)$ sending $M$ into $M'$ with $H=(F,G)\in \C^n\times \C^d$, the following identity holds for $(Z,\zeta)\in {\mathcal M}$:
\begin{equation}\label{e:jams}
\bar F_{\tau^\mu}(\zeta)\cdot Q_{\chi'}'(f(Z),\bar H(\zeta))=
{\mathcal S}_{\mu}\left(Z,\zeta,H(Z),\bar H(\zeta);\widehat j_{Z}^{|\mu|} {H}, (\bar F_{\tau^\gamma}(\zeta))_{|\gamma|\leq |\mu|-1}, (\bar G_{\tau^\eta}(\zeta))_{|\gamma|\leq |\mu|} \right).
\end{equation}
\end{lem}

\begin{proof} The proof follows easily by induction and
  differentiating \eqref{e:fundamental} with respect to $\tau$. 
  We leave the details of this to the reader.
\end{proof}

The following lemma is stated in \cite[Lemma 9.3]{LM2} for the case of
biholomorphic self-maps of real-analytic generic submanifolds but it
(along with the proof) also applies
to the case of arbitrary
formal holomorphic maps between formal generic submanifolds.

\begin{lem}\label{l:compute} Let $M,M'\subset \CN$ be formal generic submanifolds of
  codimension $d$ given in normal coordinates as above. Then for every
  multiindex $\mu \in \N^d\setminus \{0\}$, there exists a universal
  $\Cd$-valued power series mapping $W_\mu (Z,\zeta,Z',\zeta';\cdot)$ polynomial in its last argument 
   with coefficients in the ring $\C \dbl Z,\zeta,Z',\zeta' \dbr$ such
  that for every formal holomorphic map $H\colon (\CN,0)\to (\CN,0)$ sending
  $M$ into $M'$ with $H=(F,G)\in \C^n\times \C^d$ the following identity holds
  \begin{equation}\label{e:gmu} \bar G_{\tau^\mu}(\zeta)=\bar
    F_{\tau^\mu}(\zeta)\cdot \bar Q_{\chi'}'(\bar F(\zeta),{H}(Z))+ W_\mu
    \left(Z,\zeta,H(Z),\bar{H}(\zeta); \widehat j_{Z}^{|\mu|} {H},\widehat
    j_{\zeta}^{|\mu|-1} {\bar H}\right).  \end{equation} In particular, there
    exists a universal $\C^d$-valued polynomial map ${\mathcal R}_\mu={\mathcal
    R}_\mu(\chi,\chi';\cdot)$ of its arguments with coefficients in the ring
    $\C \dbl \chi,\chi' \dbr$ such that for every map $H$ as above, the
    following holds: \begin{equation}\label{e:gderivative} \bar
      G_{\tau^\mu}(\chi,0)={\mathcal R}_\mu \left(\chi,\bar
      F(\chi,0);(\partial^\beta \bar H(\chi,0))_{1\leq |\beta|\leq
      |\mu|-1},j_0^{|\mu|}H\right), \end{equation} \end{lem}

Combining Lemma \ref{l:computecomplex} and Lemma \ref{l:compute} together, we
get the following.

\begin{lem}\label{l:onemore} In the situation of Lemma~{\rm
  \ref{l:computecomplex}}, there exists, for every multiindex $\mu \in
  \N^d\setminus \{0\}$, a universal $\C^d$-valued power series mapping ${\mathcal
  A}_\mu={\mathcal A}_\mu(z,\chi,Z',\zeta',;\cdot)$ polynomial in its last argument with
  coefficients in the ring $\C \dbl z,\chi,Z',\zeta'\dbr$ such that for every
  formal holomorphic map $H\colon (\CN,0)\to (\CN,0)$ sending $M$ into $M'$
  with $H=(F,G)\in \C^n\times \C^d$ the following identity holds
  \begin{multline}\label{e:ma} \bar F_{\tau^\mu}(\chi,0)\cdot
    Q_{\chi'}'(F(z,Q(z,\chi,0)),\bar H(\chi,0))=\\ {\mathcal
    A}_{\mu}\left(z,\chi,H(z,Q(z,\chi,0)),\bar H(\chi,0); ((\partial^\beta
    H)(z,Q(z,\chi,0)))_{1\leq |\beta|\leq |\mu|}, (\partial^\beta \bar
    H(\chi,0))_{1\leq |\beta|\leq |\mu|-1}, j_0^{|\mu|}H \right).
  \end{multline}
\end{lem}

\begin{proof} Setting $Z=(z,Q(z,\chi,0))$ and $\zeta=(\chi,0)$ in
  \eqref{e:jams} and substituing $\bar G_{\tau^\eta}(\chi,0)$ by its expression
  given by \eqref{e:gderivative} yields the required conclusion of the lemma.
\end{proof}

We need a last independent lemma before proceeding with the proof of the main
proposition of this section.

\begin{lem}\label{l:ordervanish} Let $A=A(u,v)$ be a $\C^k$-valued formal
  power series mapping, $u,v\in \C^k$, satisfying ${\rm det}\, A_u(u,v)\not
  \equiv 0$ and $A(0,v)\equiv 0$. Assume that ${\rm ord}_u\, \left({\rm det}\,
  A_u(u,v)\right)\leq \nu$ for some nonnegative integer $\nu$. Then for every
  nonnegative integer $r$ and for every formal power series $\psi (t,v)\in \C
  \dbl t,v\dbr$, $t\in \C^k$, if ${\rm ord}_u \left(\psi (A(u,v),v)\right) >
  r(\nu +1)$, then necessarily ${\rm ord}_t\, \psi (t,v)> r$. 
\end{lem}

\begin{proof} We prove the lemma by induction on $r$ and notice that the
  statement automatically holds for $r=0$. Suppose that $\psi$ is as in the
  lemma and satisfies ${\rm ord}_u \left(\psi (A(u,v),v)\right) > r(\nu
  +1)$ for some $r\geq 1$.  Differentiating $\psi (A(u,v),v)$ with respect
  to $u$, we get that the order (in $u$) of each component of
  $\psi_t(A(u,v),v)\cdot A_u(u,v)$ is strictly greater than $r\nu+r-1$.
  Multiplying $\psi_t(A(u,v),v)\cdot A_u(u,v)$ by the classical inverse of
  $A_u(u,v)$, we get the same conclusion for each component  of the power series mapping $({\rm
  det}\, A_u(u,v))\, \psi_t(A(u,v),v)$. By assumption, ${\rm ord}_u\, \left({\rm det}\,
  A_u(u,v)\right)\leq \nu$ and therefore the order (in $u$) of each
  component of $\psi_t(A(u,v),v)$ is strictly greater than $r\nu+r-1-\nu=(r-1)(\nu+1)$.
  From the induction assumption, we conclude that the order in $t$ of each
  component of $\psi_t(t,v)$ (strictly) exceeds $r-1$. To conclude that ${\rm
  ord}_t\, \psi (t,v) > r$ from the latter fact, it is enough to notice
  that $\psi (0,v)\equiv 0$ since ${\rm ord}_u \left(\psi
  (A(u,v),v)\right)> r(\nu+1)\geq 1$ and $A(0,v)\equiv 0$. The proof of
  Lemma~\ref{l:ordervanish} is complete. 
\end{proof}

We are now completely ready to prove the following main goal of this section.

\begin{prop}\label{p:soso} Let $M,M'\subset \CN$ 
  be formal generic submanifolds of the same dimension, 
  given in normal coordinates as above. Assume that $M$
  belongs to the class ${\mathcal C}$ et let $k_0$ be the integer given in Corollary~{\rm \ref{c:finitejetfirst}}.  Then the integer
  $k_1:=\max \{ k_0, \kappa_M(\nu_M(\infty)+1)\}$ has the following property:
  for any  pair $H_1,H_2\colon (\CN,0)\to (\CN,0)$
  of formal CR-transversal holomorphic mappings 
  sending $M$ into $M'$ and any nonnegative integer
  $\ell$, if $j_0^{k_1+\ell} H_1=j_0^{k_1+\ell} H_2$, then
  necessarily $(\partial^\alpha H_1)(z,0)=(\partial^\alpha H_2)(z,0)$ for all
  $\alpha \in \N^N$ with $|\alpha|\leq \ell$. Furthermore,
  $k_1$ depends upper-semicontinuously on continuous deformations of $M$.
\end{prop}

\begin{proof} The proposition is proved by induction on $\ell$.  For $\ell=0$, the proposition
follows immediately from Corollary~\ref{c:finitejetfirst}. Consider now 
  a pair of maps $H_1,H_2$ as
  in the statement of the proposition with the same $k_1+\ell$ jet at $0$,
  where $\ell\geq 1$. Then
  from the induction assumption,  
  we know that $(\partial^\alpha
  H_1)(z,0)=(\partial^\alpha H_2)(z,0)$ for all $\alpha \in \N^N$ with
  $|\alpha|\leq \ell-1$. Hence it is enough to show that for all multiindices
  $\mu\in \N^d$ with $|\mu|=\ell$, 
  \begin{equation}\label{e:jewel2}
    \frac{\partial^\mu \bar H_1}{\partial \tau^\mu}(\chi,0)=\frac{\partial^\mu
    \bar H_2}{\partial \tau^\mu}(\chi,0).  
  \end{equation} 
  This is further simplified by noticing that
  Lemma \ref{l:compute} (more precisely  \eqref{e:gmu}
  applied with $Z=0$ and $\zeta=(\chi,0)$) implies
  that it is enough to prove that
  for all $\mu\in \N^d$ as above, 
  \begin{equation}\label{e:jewel}
    \frac{\partial^\mu \bar F_1}{\partial
    \tau^\mu}(\chi,0)=\frac{\partial^\mu \bar F_2}{\partial
    \tau^\mu}(\chi,0). 
  \end{equation} 
  Next, applying \eqref{e:ma} to both
  $H_1$ and $H_2$, we get the order in $z$ of each component of the power
  series mapping given by 
  \begin{equation}\label{e:hungry}
    \frac{\partial^\mu \bar F_1}{\partial \tau^\mu}(\chi,0)\cdot
    Q_{\chi'}'(F_1(z,Q(z,\chi,0)),\bar H_1(\chi,0))- \frac{\partial^\mu
    \bar F_2}{\partial \tau^\mu}(\chi,0)\cdot
    Q_{\chi'}'(F_2(z,Q(z,\chi,0)),\bar H_2(\chi,0)) 
  \end{equation} 
  is at
  least $k_1+1$. Consider the power series mapping 
  \begin{equation} 
    \psi
    (z',\chi):= \frac{\partial^\mu \bar F_1}{\partial
    \tau^\mu}(\chi,0)\cdot Q_{\chi'}'(z',\bar H_1(\chi,0))-
    \frac{\partial^\mu \bar F_2}{\partial \tau^\mu}(\chi,0)\cdot
    Q_{\chi'}'(z',\bar H_2(\chi,0)), 
  \end{equation} 
  and let $\widehat
  F(z,\chi)\in \C \dbl \chi\dbr [z]$ be the Taylor polynomial (in $z$)
  of order $k_1$ of $F_1(z,Q(z,\chi,0))$ viewed as a power series in
  the ring $\C \dbl \chi\dbr \dbl z\dbr$.  Note that it follows from
  our assumptions that  $\widehat F(z,\chi)$ coincides also with the
  Taylor polynomial (in $z$) of order $k_1$ of $F_2(z,Q(z,\chi,0))$
  (also viewed as a power series in the ring $\C \dbl \chi\dbr \dbl
  z\dbr$). Hence since the order in $z$ of each component of the power
  series mapping given by \eqref{e:hungry} is at least $k_1+1$, this
  also holds for the power series mapping $\psi (\widehat
  F(z,\chi),\chi)$.  Furthermore, we claim that
  \begin{equation}\label{e:claim} 
    {\rm ord}_z\ \left({\rm det}\,
    \widehat F_z(z,\chi)\right)\leq \nu_M (\infty).  
  \end{equation} Indeed,
  suppose not. Since $${\rm ord}_z\, \left(\widehat
  F(z,\chi)-F_1(z,Q(z,\chi,0))\right)\geq k_1+1,$$ we have
  \begin{equation}\label{e:teuf} {\rm ord}_z\,  \left(\widehat
    F_z(z,\chi)-\frac{\partial}{\partial
    z}\left[F_1(z,Q(z,\chi,0))\right] \right)\geq k_1\geq \nu_M (\infty)+1.
  \end{equation} 
  Therefore \eqref{e:teuf} yields ${\rm ord}_z\,
  \left({\rm det }\, \displaystyle \frac{\partial}{\partial
  z}\left[F_1(z,Q(z,\chi,0))\right]\right)\geq \nu_M(\infty)+1$ and hence in
  particular that 
  $${\rm ord}\, \left({\rm det} \frac{\partial
  F_1}{\partial z}(z,0)\right)\geq \nu_M(\infty)+1,$$ 
  which contradicts
  \eqref{e:unifbd} and proves the claim. Since ${\rm ord}_z\,
  \left(\psi(\widehat F(z,\chi),\chi)\right)\geq k_1+1 > 
  \kappa_M(\nu_M+1)$ and since $\widehat F(0,\chi)\equiv 0$, from
  \eqref{e:claim} and Lemma \ref{l:ordervanish} we conclude that
  ${\rm ord}_{z'}\, \psi (z',\chi) > \kappa_M$, which is equivalent
  to say that
  \begin{equation}\label{e:equal} \frac{\partial^\mu \bar
    F_1}{\partial \tau^\mu}(\chi,0)\cdot \frac{\partial
    \Theta_{\alpha}'}{\partial \chi'}(\bar F_1(\chi,0))=
    \frac{\partial^\mu \bar F_2}{\partial \tau^\mu}(\chi,0)\cdot
    \frac{\partial \Theta_{\alpha}'}{\partial \chi'}(\bar
    F_2(\chi,0)), 
  \end{equation} 
  for all $\alpha \in \N^n$ with
  $|\alpha|\leq \kappa_M$. By Corollary \ref{c:rigid}, the formal
  submanifold $M'\in {\mathcal C}$ and $\kappa_{M'}\leq \kappa_M$.
  Therefore since the formal map $\chi' \mapsto 
  \Theta_{ \kappa_M}'(\chi')$ is of generic
  rank $n$, and by assumption $\bar F_1(\chi,0)=\bar
  F_2(\chi,0)$, and since this map is not totally degenerate by virtue of Corollary~\ref{c:rigid},
  it follows from  \eqref{e:equal} that
  \eqref{e:jewel} holds 
  which completes the proof of Proposition
  \ref{p:soso}.  
\end{proof}

\subsection{Iteration and proof of Theorem \ref{t:main}}\label{ss:iter}

We now want to iterate the jet determination property along higher
order Segre sets by using the reflection identities from
\cite{LM2} established for holomorphic self-automorphisms. 
Such identities could not be used to establish
Corollary~\ref{c:finitejetfirst} and Proposition~\ref{p:soso}, since for 
CR-transversal mappings $H=(F,G)$, the matrix $F_z(0)$ need not be invertible. 
On the other hand, they will be good enough  for the iteration
process, since $F_z (z,0)$ has generic full rank in view of Corollary~\ref{c:rigid}. 
We therefore first collect from \cite{LM2} the necessary reflection
identities. Even though, as mentioned above,  such identities were considered in
\cite{LM2} only for holomorphic self-automorphisms of a given
real-analytic generic submanifold of $\CN$, we note here that
their proof also yields the same identities for merely not totally
degenerate formal holomorphic maps between formal generic
submanifolds. We start with the following version of
\cite[Proposition 9.1]{LM2}.

\begin{prop}\label{p:reflection1}
In the situation of Lemma~{\rm \ref{l:computecomplex}}, there
exists a universal power series $\D=\D(Z,\zeta; \cdot)$ polynomial in its last 
argument with coefficients in the ring $\C \dbl Z,\zeta \dbr$
and, for every $\alpha \in \N^n\setminus \{0\}$, another universal
$\Cd$-valued  power series mapping
$\PP_{\alpha}=\PP_{\alpha}(Z,\zeta;\cdot)$ $($whose components belong to the same ring as that of $\D)$, such
that for every not totally degenerate formal holomorphic map
$H\colon (\CN,0)\to (\CN,0)$ sending $M$ into $M'$ with
$H=(F,G)\in \C^n\times \C^d$ the following holds:
\begin{enumerate}
\item[{\rm (i)}] $\D (Z,\zeta;\widehat j_{\zeta}^{1} \bar{H})|_{(Z,\zeta)=(0,(\chi,0))}={\rm det}\,
\left(\bar F_\chi (\chi,0)\right)\not \equiv 0$;
\item[{\rm (ii)}]
$(\D(Z,\zeta;\widehat j_{\zeta}^{1} \overline{H}
))^{2|\alpha|-1}\,
\bar{Q}_{{\chi'}^{\alpha}}'(\bar{F}(\zeta),H(Z))=\PP_{\alpha}(Z,\zeta;\widehat
j_{\zeta}^{\left| \alpha \right|} \bar{H}),\quad {\rm for}\, \, \,
(Z,\zeta)\in {\mathcal M}$.
\end{enumerate}
\end{prop}

We also need the following version of \cite[Proposition 9.4]{LM2}.

\begin{prop}\label{p:reflection3}
In the situation of Lemma~{\rm \ref{l:computecomplex}}, for any
$\mu \in \N^d\setminus \{0\}$ and $\alpha \in \N^n\setminus\{0\}$,
there exist universal $\Cd$-valued power series mappings ${\mathcal
B}_{\mu,\alpha}(Z,\zeta,Z',\zeta';\cdot)$ and ${\mathcal
Q}_{\mu,\alpha}(Z,\zeta;\cdot)$ polynomial in their last argument with
coefficients in the ring $\C \dbl Z,\zeta,Z',\zeta' \dbr$ and $\C
\dbl Z,\zeta \dbr $ respectively such that for every not totally
degenerate formal holomorphic map $H\colon (\CN,0)\to (\CN,0)$
sending $M$ into $M'$ with $H=(F,G)\in \C^n\times \C^d$ the
following holds:
\begin{equation}\label{e:fundamental3}
F_{w^\mu}(Z)\cdot
\left(\bar{Q}_{{\chi'}^{\alpha},z'}'(\bar{F}(\zeta),H(Z))+
Q_{z'}'(F(Z),\bar{H}(\zeta))\cdot
\bar{Q}_{{\chi'}^{\alpha},w'}(\bar{F}(\zeta),H(Z))\right)=(*)_1+(*)_2,
\end{equation}
where $(*)_1$  is given by
\begin{equation}\label{e:*1}
(*)_1:={\mathcal B}_{\mu,\alpha}\left(Z,\zeta,H(Z),\bar{H}(\zeta);
\widehat j_{Z}^{\left| \mu \right| - 1} H, \widehat
j_{\zeta}^{\left| \mu \right|}{\bar{H}} \right),
\end{equation}
and $(*)_2$ is given by
\begin{equation}\label{e:*2}
(*)_2:=\frac{{\mathcal Q}_{\mu,\alpha}(Z,\zeta, \widehat
j_{\zeta}^{|\alpha| + |\mu|} \bar{H})}{(\D(Z,\zeta, \widehat
j_{\zeta}^{1} \bar{H} ))^{2|\alpha|+|\mu|-1}},
\end{equation}
and where ${\mathcal D}$ is given by Proposition~{\rm
\ref{p:reflection1}}.
\end{prop}

In what follows, we use the notation introduced for the Segre
mappings given in \S \ref{ss:segremaps} (associated to a fixed choice of normal coordinates for $M$). 
We are now ready to prove the following.

\begin{prop}\label{p:iteration}
Let $M,M'$ be formal generic submanifolds of $\CN$ of the same
dimension given in normal coordinates as above. Assume that $M'$
belongs to the class ${\mathcal C}$ and let $j$ be a positive
integer.  Then for every nonnegative integer $\ell$  and for every
pair $H_1,H_2\colon (\CN,0)\to (\CN,0)$ of not totally degenerate
formal holomorphic mappings sending $M$ into $M'$, if
$(\partial^\alpha H_1)\circ v^j= (\partial^\alpha H_2)\circ v^j$
for all $\alpha \in \N^N$ with $|\alpha|\leq \kappa_{M'}+\ell$
then necessarily for all $\beta \in \N^N$ with $|\beta|\leq \ell$,
one has
$$(\partial^\beta H_1)\circ v^{j+1}= (\partial^\beta H_2)\circ v^{j+1}.$$
\end{prop}

\begin{proof} We prove the proposition by induction on $\ell$.

For  $\ell =0$, suppose that $H_1,H_2\colon (\CN,0)\to (\CN,0)$ is
a pair of not totally degenerate formal holomorphic mappings
sending $M$ into $M'$ satisfying
\begin{equation}\label{e:explain}
(\partial^\alpha H_1)\circ v^j= (\partial^\alpha H_2)\circ
v^j,\quad  |\alpha|\leq \kappa_{M'}.
\end{equation}
 Then setting $Z=v^{j+1}(t^{[j+1]})$ and
$\zeta=\bar{v}^{j}(t^{[j]})$ in Proposition \ref{p:reflection1}
(ii) and using the above assumption, one obtains that for all
$\alpha \in \N^N$ with $|\alpha|\leq \kappa_{M'}$
\begin{equation}\label{e:equalityQ}
\bar{Q}_{{\chi'}^{\alpha}}'(\bar{F}_1\circ \bar v^{j},H_1\circ
v^{j+1})= \bar{Q}_{{\chi'}^{\alpha}}'(\bar{F}_2\circ \bar
v^{j},H_2 \circ v^{j+1}).
\end{equation}
In what follows, to avoid some unreadable notation, we denote by
$V^j=V^{j}(T^1,\ldots,T^{j+1})$ the Segre mapping of order $j$
associated to $M'$ and also write $T^{[j]}=(T^1,\ldots,T^{j})\in
\C^n\times \ldots \times \C^n$. Next we note that we also have
\begin{equation}\label{e:piano}
H_\nu\circ v^{j+1}=V^{j+1}({F}_\nu\circ v^{j+1},\bar F_\nu\circ
\bar v^{j}, F_\nu \circ v^{j-1},\ldots),\quad \nu=1,2.
\end{equation}
 Since $M'\in {\mathcal C}$, we may choose multiindices
 $\alpha^{(1)},\ldots,\alpha^{(n)}\in \N^n$ and $s_1,\ldots,s_n\in \{1,\ldots,d\}$ with $|\alpha^{j}|\leq \kappa_{M'}$
  such that the formal map 
  $\bar \Theta' \colon z' 
  \mapsto ({{\bar{\Theta}}}^{'s_1}_{\alpha^{(1)}}(z'),
  \ldots,{{\bar{\Theta}}}^{'s_n}_{\alpha^{(n)}}(z'))$
  is of generic rank $n$. Denote by $\Psi$ the formal map $(T^{j+1},\ldots,T^{1})\mapsto
   \left(\bar{Q}^{'s_i}_{{\chi'^{\alpha^{(i)}}}}(T^{j},V^{j+1}(T^{j+1},T^j,\ldots,T^1))\right)_{1\leq i\leq
   n}$. As in the proof of Corollary \ref{c:finitejetfirst}, we
   write
   $$\Psi (u,T^{[j]})-\Psi (v,T^{[j]})=(u-v)\cdot \int_0^1 \Psi_{T^{j+1}}(tu+(1-t)v,T^{[j]})dt,$$
and note that it follows from \eqref{e:equalityQ}, \eqref{e:piano}
and \eqref{e:explain} that
\begin{align*}
0&=\Psi (F_1\circ v^{j+1},\bar F_1\circ \bar
v^j,\ldots)-\Psi (F_2\circ v^{j+1},\bar F_2\circ \bar v^j,\ldots)\\
&=(F_1\circ v^{j+1}-F_2\circ v^{j+1})\cdot \int_0^1
\Psi_{T^{j+1}}(tF_1\circ v^{j+1}+(1-t)F_2\circ v^{j+1},\bar
F_1\circ \bar v^j,F_1\circ v^{j-1},\ldots)\, dt.
\end{align*}
We now claim that ${\rm det}\, \int_0^1 \Psi_{T^{j+1}}(tF_1\circ
v^{j+1}+(1-t)F_2\circ v^{j+1},\bar F_1\circ \bar v^j,\ldots)dt
\not \equiv 0$. Indeed if it were not the case  we would in
particular have, after setting $t^{[j]}=0$ in the above
determinant, that
\begin{equation}\label{e:det}
 {\rm det}\, \left(\frac{\partial \bar \Theta'}{\partial z'}(F_1\circ
v^1)\right)\equiv 0. \end{equation} But since $H_1$ is not totally
degenerate,  \eqref{e:det} implies that $\Rk \bar \Theta'<n$, a
contradiction. This proves the claim and hence that $F_1\circ
v^{j+1}=F_2\circ v^{j+1}$ and therefore that $H_1\circ
v^{j+1}=H_2\circ v^{j+1}$ in view of \eqref{e:piano}. This
completes the proof of the proposition for the case $\ell =0$.

Now assume that $\ell>0$ and suppose that $(\partial^\alpha
H_1)\circ v^j= (\partial^\alpha H_2)\circ v^j$ for all $\alpha \in
\N^N$ with $|\alpha|\leq \kappa_{M'}+\ell$. From the induction
assumption, we know that
\begin{equation}\label{e:explainagain}
(\partial^\beta H_1)\circ v^{j+1}= (\partial^\beta H_2)\circ
v^{j+1},\quad \forall \beta \in \N^N,\, \, \, |\beta|\leq \ell-1.
\end{equation}
It remains therefore to show the equality of the $\ell$-th
order derivatives restricted to the $j+1$-th Segre set. We first
prove that this is so for the pure transversal derivatives i.e.
that
\begin{equation}\label{e:yield}
\forall \mu \in \N^d,\ |\mu|=\ell,\quad
\frac{\partial^{|\mu|}H_1}{
\partial w^\mu}\circ v^{j+1}=
 \frac{\partial^{|\mu|}H_2}{ \partial w^\mu}\circ v^{j+1}.
\end{equation}
Let $\mu$ be such a multiindex. Setting $Z=v^{j+1}(t^{[j+1]})$ and
$\zeta=\bar v^{j}(t^{[j]})$ in \eqref{e:fundamental3} applied to
both $H_1$ and $H_2$ and using \eqref{e:explainagain}, we get for
all $\alpha \in \N^n$ with $|\alpha|\leq \kappa_{M'}$
\begin{equation}\label{e:earth}
\left(\frac{\partial^{|\mu|}F_1}{ \partial w^\mu}\circ
v^{j+1}\right)\cdot \Gamma_\alpha
=\left(\frac{\partial^{|\mu|}F_2}{
\partial w^\mu}\circ v^{j+1}\right)\cdot \Gamma_\alpha,
\end{equation}
where
$$\Gamma_\alpha=\Gamma_\alpha (t^{[j+1]}):=\bar{Q}_{{\chi'}^{\alpha},z'}'(\bar{F}_1\circ \bar v^{j},H_1\circ v^{j+1})+
Q_{z'}'(F_1\circ v^{j+1},\bar{H}_1\circ \bar v^{j})\cdot
\bar{Q}_{{\chi'}^{\alpha},w'}(\bar{F}\circ \bar v^j,H\circ
v^{j+1}).$$ To conclude from \eqref{e:earth} that
 $\displaystyle \frac{\partial^{|\mu|}F_1}{ \partial w^\mu}\circ v^{j+1}=\displaystyle
 \frac{\partial^{|\mu|}F_2}{ \partial w^\mu}\circ v^{j+1}$, it is enough to show that the
  generic rank of the family of matrices  $(\Gamma_\alpha)_{|\alpha|\leq \kappa_{M'}}$ is
  $n$. This holds trivially since $\Gamma_\alpha(0,\ldots,0,t^{j+1})=
  \displaystyle \frac{\partial \bar \Theta_{\alpha}'}{\partial z'}(t^{j+1})$
  and since $M'\in {\mathcal C}$. Next using the identity \eqref{e:gmu} given by
  Lemma \ref{l:compute} applied to $\zeta=\bar v^{j+1}(t^{[j+1]})$ and
  $Z=v^{j}(t^{[j]})$, we immediately get that
 $\displaystyle \frac{\partial^{|\mu|}\bar G_1}{ \partial \tau^\mu}\circ \bar v^{j+1}=\displaystyle
 \frac{\partial^{|\mu|}\bar G_2}{ \partial \tau^\mu}\circ \bar v^{j+1}$ which
 yields \eqref{e:yield}.

 To complete the proof of the induction, we need to show that $(\partial^\beta H_1)\circ v^{j+1}=
(\partial^\beta H_2)\circ v^{j+1}$ for arbitrary
$\beta=(\beta_1,\ldots,\beta_N) \in \N^N$ with $|\beta|=\ell$. We
prove it by induction on the number
$c_\beta:=\beta_1+\ldots+\beta_n$. For $c_\beta=0$, this follows
from \eqref{e:yield} proved above. Now if $c_\beta>0$, we may
assume without loss of generality that $\beta_1>0$ and write
$\beta=(1,0,\ldots,0)+\tilde \beta$ with $|\tilde \beta|=\ell-1$.
By \eqref{e:explainagain} we know that $\partial^{\tilde
\beta}H_1\circ v^{j+1}=\partial^{\tilde \beta}H_2\circ v^{j+1}$
and hence by differentiating this latter identity with respect to
first variable of $t^{j+1}\in \C^n$, we get
\begin{multline}\label{e:sickofthiscrap?}
\partial^{\beta}H_1\circ v^{j+1}+Q_{z_1}(t^{j+1},\bar v^{j})\cdot
\left( \left(\frac{\partial^{\ell}H_1}{\partial w \partial
z^{\tilde \beta}}\right)\circ v^{j+1}\right)\\=
\partial^{\beta}H_2\circ v^{j+1}+Q_{z_1}(t^{j+1},\bar v^{j})\cdot
\left(\left(\frac{\partial^{\ell}H_2}{\partial w \partial
z^{\tilde \beta}}\right)\circ v^{j+1}\right),
\end{multline}
from which the desired equality  $\partial^{\beta}H_1\circ
v^{j+1}=\partial^{\beta}H_2\circ v^{j+1}$ follows by using the
induction assumption. The proof of the proposition is therefore
complete.
\end{proof}

Combining now Proposition \ref{p:iteration}, Proposition~\ref{p:soso}, Corollary
\ref{c:rigid} and Corollary \ref{c:finitejetfirst}, one gets the
following.

\begin{prop}\label{p:almostdone}
Let $M,M'\subset \CN$ be formal generic submanifolds of 
the same dimension. Assume that $M$
belongs to the class ${\mathcal C}$. Then for every positive
integer $j$, the integer
\[
k_j = k_1+\kappa_M(j-1),\]
where $k_1$ is the integer defined in Proposition~{\rm \ref{p:soso}},
has the following property: 
If $H_1,H_2\colon (\CN,0)\to (\CN,0)$ are two formal
CR-transversal holomorphic mappings sending $M$ into $M'$ 
such that $j_0^{k_j} H_1 = j_0^{k_j} H_2$, 
then necessarily $H_1\circ
v^j=H_2\circ v^j $. Furthermore, $k_j$ depends
upper-semicontinuously on continuous deformations of $M$.
\end{prop}

\begin{proof}{Completion of the proof of Theorem~\ref{t:main}}
Firstly, we may assume that $M$ and $M'$ are given in normal
coordinates as above and we denote by $d$ the codimension of $M$.
By Proposition \ref{p:almostdone},  
if $H_1,H_2\colon (\CN,0)\to
(\CN,0)$ are two formal CR-transversal holomorphic mappings
sending $M$ into $M'$ with the same $k_{d+1}$-jet,
then necessarily $H_1\circ v^{d+1}=H_2\circ v^{d+1}$. By the
finite type assumption on $M$, we have from Theorem \ref{t:ber}
that $\Rk v^{d+1}=N$  and hence that $H_1=H_2$. Furthermore, it also follows from
Proposition~\ref{p:almostdone} that the integer $k_{d+1}$ depends upper-semicontinuously on perturbations of
$M$, which completes the proof of Theorem~\ref{t:main}.
\end{proof}

\section{Application to smooth generic submanifolds and proofs of Theorem~\ref{t:main1} and
Corollary~\ref{c:cor1}}\label{s:end}

We have the following result obtained from Theorem~\ref{t:main} by considering the smooth deformation of $M$ given by varying its base point as explained in \S \ref{ss:second}.

\begin{thm}\label{t:main22}
Let $M\subset \C^N$ be a smooth generic submanifold that is in the
class ${\mathcal C}$ and of finite type at each of its points.
Then for every point $p\in M$ there exists an integer $\ell_p$,
depending upper-semicontinuously on $p$, such that for every
smooth generic submanifold $M'\subset \CN$ of the same dimension
as that of $M$, if $h_1,h_2\colon (M,p)\to M'$ are two germs of
smooth CR-transversal mappings with the same $\ell_p$ jet at $p$,
then necessarily $j^k_ph_1=j_p^kh_2$ for all positive integers
$k$.
\end{thm}

We also have the following slightly stronger 
version of Corollary~\ref{c:cor1} which is an immediate consequence of Corollary~\ref{c:cor2} and the fact that
any smooth real hypersurface of $\CN$ that is of D'Angelo finite type at some point $p\in M$ necessarily does not contain any formal curve at that point.

\begin{cor}\label{c:newcorol}
  Let $M,M'\subset \CN$ be smooth real hypersurfaces. Assume that 
  $M\in \mathcal{C}$
and that $M'$ is of D'Angelo finite type at each of their points.
Then for every $p\in M$, 
there exists an integer $\ell=\ell(p)$, 
depending upper-semicontinuously on $p$, such that 
 if $h_1,h_2\colon (M,p)\to M'$ are two germs of 
 smooth CR mappings with the same $\ell$-jet at $p$,
then necessarily $j^k_ph_1=j^k_ph_2$ for all positive integers $k$.
\end{cor}

\begin{proof}{Proof of Theorem~\ref{t:main1}}
Theorem~\ref{t:main1} follows immediately from
Theorem~\ref{t:main22} since in the setting of
Theorem~\ref{t:main1} smooth CR finite mappings are automatically
CR-transversal (see \cite{ER1}) and since every germ of an
essentially finite smooth generic submanifold of $\CN$ is
necessarily in the class ${\mathcal C}$.
\end{proof}

\begin{proof}{Proof of Corollary~\ref{c:cor1}} Corollary~\ref{c:cor1} follows immediately from Corollary~\ref{c:newcorol} since any smooth real hypersurface of $\CN$ that is of D'angelo finite at some point $p\in M$ is necessarily in the class ${\mathcal C}$ at that point. (We note here that Corollary~\ref{c:cor1} could also be derived directly from Theorem~\ref{t:main1} using some results from \cite{BR4,ER1}.) 
The last part of the corollary follows from the first part after 
applying the regularity result 
given in \cite{DP2} (see also
\cite{H2} for the case $N=2$).
\end{proof}

\bibliographystyle{hplain}
\bibliography{LM3}

\begin{thebibliography}{10}

\bibitem{BER1}
M.~S. Baouendi, P.~Ebenfelt, and L.~P. Rothschild.
\newblock C{R} automorphisms of real analytic manifolds in complex space.
\newblock {\em Comm. Anal. Geom.}, 6(2):291--315, 1998, math.CV/9603201.

\bibitem{BER4}
M.~S. Baouendi, P.~Ebenfelt, and L.~P. Rothschild.
\newblock Rational dependence of smooth and analytic {C}{R} mappings on their
  jets.
\newblock {\em Math. Ann.}, 315:205--249, 1999, math.CV/9811104.

\bibitem{BERbook}
M.~S. Baouendi, P.~Ebenfelt, and L.~P. Rothschild.
\newblock {\em Real submanifolds in complex space and their mappings}.
\newblock Princeton University Press, Princeton, NJ, 1999.

\bibitem{BER5}
M.~S. Baouendi, P.~Ebenfelt, and L.~P. Rothschild.
\newblock Convergence and finite determination of formal {{C}{R}} mappings.
\newblock {\em J. Amer. Math. Soc.}, 13(4):697--723 (electronic), 2000,
  math.CV/9904085.

\bibitem{BERbull}
M.~S. Baouendi, P.~Ebenfelt, and L.~P. Rothschild.
\newblock Local geometric properties of real submanifolds in complex space.
\newblock {\em Bull. Amer. Math. Soc. (N.S.)}, 37(3):309--336 (electronic),
  2000.

\bibitem{BMR1}
M.~S. Baouendi, N.~Mir, and L.~P. Rothschild.
\newblock Reflection ideals and mappings between generic submanifolds in
  complex space.
\newblock {\em J. Geom. Anal.}, 12(4):543--580, 2002, math.CV/0012243.

\bibitem{BR4}
M.~S. Baouendi and L.~P. Rothschild.
\newblock Geometric properties of mappings between hypersurfaces in complex
  space.
\newblock {\em J. Differential Geom.}, 31(2):473--499, 1990.

\bibitem{Be1}
V.K. Beloshapka.
\newblock On the dimension of the group of automorphisms of an analytic
  hypersurface.
\newblock {\em Math. Notes}, 14:223--245, 1980.

\bibitem{Be2}
V.K. Beloshapka.
\newblock A uniqueness theorem for automorphisms of a nondegenerate surface in
  a complex space.
\newblock {\em Math. Notes}, 47(3):239--242, 1990.

\bibitem{BG1}
T.~Bloom and I.~Graham.
\newblock On ``type'' conditions for generic real submanifolds of
  $\mathbb{C}^n$.
\newblock {\em Invent. Math.}, 40(3):217--243, 1977.

\bibitem{BK}
D.~Burns and S.~G. Krantz.
\newblock Rigidity of holomorphic mappings and a new {S}chwarz lemma at the
  boundary.
\newblock {\em J. Amer. Math. Soc.}, 7(3):661--676, 1994.

\bibitem{Ca1}
E.~Cartan.
\newblock Sur la g\'eom\'etrie pseudo-conforme des hypersurfaces de deux
  variables complexes {I}.
\newblock {\em Ann. Math. Pura Appl.}, 11(4):17--90, 1932.

\bibitem{Ca2}
E.~Cartan.
\newblock Sur la g\'eom\'etrie pseudo-conforme des hypersurfaces de deux
  variables complexes {II}.
\newblock {\em Ann. Scuola Norm. Sup. Pisa}, 1(2):333--354, 1932.

\bibitem{Hca}
H.~Cartan.
\newblock {\em Sur les groupes de transformations analytiques.}
\newblock Act.~Sc.~et~Int. Hermann, Paris, 1935.

\bibitem{CM}
S.~S. Chern and J.~K. Moser.
\newblock Real hypersurfaces in complex manifolds.
\newblock {\em Acta Math.}, 133:219--271, 1974.

\bibitem{DAfintyp}
J.~P. D'Angelo.
\newblock Real hypersurfaces, orders of contact, and applications.
\newblock {\em Ann. of Math. (2)}, 115(3):615--637, 1982.

\bibitem{DF2}
K.~Diederich and J.~E. Forn{\ae}ss.
\newblock Pseudoconvex domains with real-analytic boundary.
\newblock {\em Ann. of Math. (2)}, 107(2):371--384, 1978.

\bibitem{DF1}
K.~Diederich and J.~E. Forn{\ae}ss.
\newblock Proper holomorphic mappings between real-analytic pseudoconvex
  domains in {${\mathbb{C}}\sp n$}.
\newblock {\em Math. Ann.}, 282(4):681--700, 1988.

\bibitem{DP2}
K.~Diederich and S.~Pinchuk.
\newblock Regularity of continuous {CR} maps in arbitrary dimension.
\newblock {\em Michigan Math. J.}, 51(1):111--140, 2003; Erratum in: {\em
  Michigan Math. J}, 51(3):667--668, 2003.

\bibitem{E4}
P.~Ebenfelt.
\newblock Finite jet determination of holomorphic mappings at the boundary.
\newblock {\em Asian J. Math.}, 5(4):637--662, 2001, math.CV/0001116.

\bibitem{ELZ1}
P.~Ebenfelt, B.~Lamel, and D.~Zaitsev.
\newblock Finite jet determination of local analytic {CR} automorphisms and
  their parametrization by 2-jets in the finite type case.
\newblock {\em Geom. Funct. Anal.}, 13(3):546--573, 2003, math.CV/0107013.

\bibitem{ER1}
Peter Ebenfelt and Linda~P. Rothschild.
\newblock Transversality of {CR} mappings.
\newblock {\em Amer. J. Math.}, 128(5):1313--1343, 2006.

\bibitem{H6}
X.~Huang.
\newblock Some applications of {B}ell's theorem to weakly pseudoconvex domains.
\newblock {\em Pacific J. Math.}, 158(2):305--315, 1993.

\bibitem{H5}
X.~Huang.
\newblock A boundary rigidity problem for holomorphic mappings on some weakly
  pseudoconvex domains.
\newblock {\em Canad. J. Math.}, 47(2):405--420, 1995.

\bibitem{H2}
X.~Huang.
\newblock Schwarz reflection principle in complex spaces of dimension two.
\newblock {\em Comm. Partial Differential Equations}, 21(11-12):1781--1828,
  1996.

\bibitem{H4}
X.~Huang.
\newblock Local equivalence problems for real submanifolds in complex spaces.
\newblock In {\em Real methods in complex and CR geometry, Lecture Notes in
  Math., 1848}, pages 109--163. Springer, Berlin, 2004.

\bibitem{kim}
S.-Y. Kim.
\newblock Complete system of finite order for {CR} mappings between real
  analytic hypersurfaces of degenerate {L}evi form.
\newblock {\em J. Korean Math. Soc.}, 38(1):87--99, 2001.

\bibitem{KZ1}
S.~Y. Kim and D.~Zaitsev.
\newblock Equivalence and embedding problems for {CR}-structures of any
  codimension.
\newblock {\em Topology}, 44(3):557--584, 2005.

\bibitem{Ko1}
J.~J. Kohn.
\newblock Boundary behavior of $\bar{\partial}$ on weakly pseudo-convex
  manifolds of dimension two.
\newblock {\em J. Differential Geometry}, 6:523--542, 1972.

\bibitem{Travis1}
R.~T. Kowalski.
\newblock A hypersurface in {$\mathbb{ C}^2$} whose stability group is not
  determined by 2-jets.
\newblock {\em Proc. Amer. Math. Soc.}, 130(12):3679--3686 (electronic), 2002,
  math.CV/0107170.

\bibitem{LM4}
B.~Lamel and N.~Mir.
\newblock Remarks on the rank properties of formal {CR} maps.
\newblock {\em Science in China Series A}, 49(11):1477--1490, 2006,
  math.CV/0608048.

\bibitem{LM6}
B.~Lamel and N.~Mir.
\newblock Finite jet determination of local {C}{R} automorphisms through
  resolution of degeneracies.
\newblock {\em Asian J. Math.}, 11(2):201--216, June 2007, math.CV/0609380.

\bibitem{LM2}
B.~Lamel and N.~Mir.
\newblock Parametrization of local {CR} automorphims by finite jets and
  applications.
\newblock {\em J. Amer. Math. Soc.}, 20:519--572, 2007, math.CV/0608047.

\bibitem{Lo}
A.~V. Loboda.
\newblock On local automorphisms of real-analytic hypersurfaces.
\newblock {\em Math. USSR, Izv.}, 18:537--559, 1982.

\bibitem{Me1}
F.~Meylan.
\newblock A reflection principle in complex space.
\newblock {\em Indiana Univ. Math. J.}, 44:783--796, 1995.

\bibitem{Rsurvey}
L.P. Rothschild.
\newblock Mappings between real submanifolds in complex space.
\newblock In {\em Explorations in complex and Riemannian geometry}, pages
  253--266. Amer. Math. Soc., Providence, RI, 2003.

\bibitem{Ta1}
N.~Tanaka.
\newblock On the pseudo-conformal geometry of hypersurfaces of the space of $n$
  complex variables.
\newblock {\em J. Math. Soc. Japan}, 14:397--429, 1962.

\bibitem{Vi}
A.~G. Vitushkin.
\newblock Holomorphic mappings and the geometry of hypersurfaces.
\newblock In {\em Encyclopaedia of Mathematical Sciences, Vol. 7, Several
  Complex Variables I}, pages 159--214. Springer-Verlag, Berlin, 1985.

\bibitem{Zsurvey}
D.~Zaitsev.
\newblock Unique determination of local {C}{R}-maps by their jets: {A} survey.
\newblock {\em Rend. Mat. Acc. Lincei, s. 9}, 2002.

\end{thebibliography}
\end{document}